\documentclass[12pt, a4paper]{article}
\usepackage{arxiv}
\usepackage[centertags]{amsmath}
\usepackage{amsfonts}
\usepackage{amssymb}
\usepackage{amscd}
\usepackage{algorithmic}
\usepackage{float}
\usepackage{tikz}
\usetikzlibrary{decorations.pathreplacing}
\usetikzlibrary{arrows.meta}
\usepackage{tikz-cd}
\usepackage{mdframed}
\usepackage[T1]{fontenc}
\numberwithin{equation}{section}
\newtheorem{theorem}{Theorem}[section]
\newtheorem{lemma}{Lemma}[section]
\newtheorem{remark}{Remark}[section]

\newtheorem{proposition}{Proposition}[section]

\newcommand{\DDS}{\mathbf{T}}
\newcommand {\mabs} [1]{\mbox{$\left\bracevert\!\! #1 \!\! \right\bracevert\!$}} 
\newcommand{\vertiii}[1]{{\left\vert\kern-0.25ex\left\vert\kern-0.25ex\left\vert #1 
		\right\vert\kern-0.25ex\right\vert\kern-0.25ex\right\vert}}
\newcommand{\vertiiis}[1]{{\vert\kern-0.25ex\vert\kern-0.25ex\vert #1 
		\vert\kern-0.25ex\vert\kern-0.25ex\vert}}	
\def\vecop {\mathop{\rm vec}\nolimits}	
\newcommand{\Koop}{\mathcal{K}}
\newcommand{\KoopM}{\mathsf{K}}	
\newcommand{\dmd}{\textsf{DMD}}
\newcommand{\edmd}{\textsf{EDMD}}
\newcommand{\kmd}{\textsf{KMD}}
\newcommand{\x}{\mathbf x}
\newcommand{\y}{\mathbf y}

\newcommand{\0}{\mathbf{0}}
\newcommand{\zbA}{\mathsf A}

\newcommand{\zbH}{\mathsf H}

\newcommand{\zbL}{\mathsf L}

\newcommand{\zbM}{\mathsf M}

\newcommand{\zbR}{\mathsf R}
\newcommand{\zbr}{\mathsf r}
\newcommand{\zbW}{\mathsf W}
\newcommand{\zbx}{\mathsf x}
\newcommand{\zbz}{\mathsf z}
\newcommand{\zbF}{\mathsf F}

\newcommand{\zbG}{\mathsf G}

\newcommand{\zbQ}{\mathsf Q}

\newcommand{\zbu}{\mathsf u}

\newcommand{\zbv}{\mathsf v}

\newcommand{\zby}{\mathsf y}
\newcommand{\R}{\mathbb{R}}
\newcommand{\FF}{\mathcal F}
\newcommand{\qr}{\textsf{QR}}
\newcommand{\svd}{\textsf{SVD}}
\newcommand{\basef}{{\boldsymbol\wp}}
\newcommand{\Basef}{{\boldsymbol\varPi}}
\newcommand{\bfkappa}{{\boldsymbol\kappa}}
\newcommand{\uuk}{\bfkappa}
\newcommand{\OX}{\mathsf B_x}

\newcommand{\OY}{\mathsf B_y}

\newcommand{\bfPhi}{{\boldsymbol\Phi}}
\newcommand{\C}{\mathbb{C}}

\newcommand{\Id}{\mathbb I}
\newcommand{\bfr}{\mathbf r}
\newcommand{\bfR}{\mathbf R}
\newcommand{\bfSS}{\mathbf{Q}}

\newcommand{\bfss}{\mathbf{q}}

\def\rank {\mathop{\rm rank}\nolimits}
\newcommand{\resolvent}{{\boldsymbol\rho}}
\newcommand{\bfsigma}{{\boldsymbol\sigma}}
\newcommand{\bfdelta}{{\boldsymbol\delta}}
\newcommand{\bftau}{{\boldsymbol\tau}}
\newcommand{\GramM}{\boldsymbol{\Gamma}}
\newcommand{\ei}{\mathsf{e}}
\title{On residual bounds of the EDMD solution to the eigenvalue problem for the Koopman operator and backward shadowing stability of the EDMD/KMD}
\author{Zlatko Drma\v{c}\thanks{University of Zagreb, Faculty of Science, Department of Mathematics, Bijeni\v{c}ka 30, 10000 Zagreb, Croatia. The author is supported by the Croatian Science Foundation (HRZZ) Grant IP-2025-02-3733 \emph{``Data driven identification and dimension reduction of dynamical systems''}
and in parts by the European Union – NextGenerationEU through the National Recovery and Resilience Plan 2021-2026,  institutional grant of University of Zagreb Faculty of Science (IK IA 1.1.3. Impact4Math) and the \emph{``Implementation of cutting-edge research and its application as part of the Scientific Center of Excellence for Quantum and Complex Systems, and Representations of Lie Algebra''},  Grant No. PK.1.1.10.0004, co-financed by the European Union through the European Regional Development Fund - Competitiveness and Cohesion Programme 2021-2027.}}

\begin{document}

\maketitle

\begin{abstract}
This paper introduces shadowing theory based backward stability analysis of data driven computational
study of discrete dynamical systems $\x_{k+1}=\DDS(\x_k)$ in the framework of the Koopman composition operator $\Koop$, the Extended Dynamic Mode Decomposition (\edmd{}) and the Koopman Mode Decomposition (\kmd{}). 
The proposed framework establishes the missing theoretical link between operator approximation error in \edmd{} and trajectory--level accuracy through the theory of shadowing.
Data driven spectral analysis of the dynamics using the approximate (computed) eigenpairs of $\Koop$ is cast in terms of the backward error analysis. The individual residuals $\zbr_i=\Koop\phi_i-\lambda_i\phi_i=\phi_i\circ\DDS-\lambda_i\phi_i$ of the computed eigenpairs $(\lambda_i,\phi_i)$, $i=1,2, \ldots$, are fused into a backward perturbation $\Delta\Koop$, so that the computed eigenpairs are exact for $\Koop-\Delta\Koop$: $(\Koop-\Delta\Koop)\phi_i=\lambda_i\phi_i$. Data driven formula for computing approximate operator norm of $\Delta\Koop$ is provided. Then, the impact of $\Delta\Koop$ is carried over to the map $\DDS$ and the initial condition $\x_1$, showing that the computed approximations and \kmd{} based forecast of an observable of interest correspond exactly to its values along a pseudo--trajectory of $\DDS$. In the final step,  the pseudo--trajectory is shadowed by an exact trajectory of $\DDS$.  This interpretation of errors is in particular suitable in data driven scenarios where the data is contaminated by noise.
New insights into numerical shadowing and an extension of the numerical shadowing lemma are presented.
\end{abstract}

\section{Introduction}
Many phenomena in applied sciences and engineering are modeled as discrete dynamical system
\begin{equation}\label{eq:DDS}
	\x_{k+1}=\DDS(\x_k), 
\end{equation}
where $\DDS : \mathcal{X} \longrightarrow \mathcal{X}$ is a map on a state space $\mathcal{X}\subseteq\R^n$.   The mapping $\DDS$ may or may not be known. If it is known, then many trajectories can be generated with various initial conditions, e.g. by numerical simulations of semi--discretized partial differential equations, and the structure of the dynamics can be studied. This can provide valuable inputs for real world applications as well as for theoretical considerations. In another setting, the trajectory data is experimental. Examples  include quantification of images obtained from  high speed camera recordings of a combustion process in a turbine \cite{Yang-combustion-2020}, 
quantitative measurement of unsteady surface pressure distribution on aerodynamic bodies using fast Pressure--Sensitive Paint (PSP) optical sensors \cite{Ali-sensitive-paint-2016}, wind tunnel Particle Image Velocimentry (PIV) measurements \cite{Rowley-etal-PIVdata-2014}, numbers of cases of COVID--19 infections, recoveries and deaths, reported daily in administrative units of a state (cities, counties) \cite{Mezic-etal-Covid19-2024}. Here, the goal can be to identify the dynamics from the data, e.g. for model predictive control. 

The dynamics can be studied using the observables of the system, i.e. functions of the states $f(\x_1)$, $f(\x_2)=f(\DDS(\x_1))$, $f(\x_3)=f(\DDS(\x_2))=f(\DDS(\DDS(\x_1))), \ldots$ of (\ref{eq:DDS}). This can be formalized by choosing the observables from a suitable Hilbert space $\FF$ and defining the composition operator $\Koop$ with the domain $\mathcal{D}(\Koop)\subseteq\FF$ by $\Koop f = f\circ\DDS$. 
Some details, depending on the properties of the mapping $\DDS$ and the underlying space $\FF$, are needed to ensure that $\Koop$ is well defined and that $\Koop f\in\mathcal{D}(\Koop)$, so that the powers $\Koop f, \Koop^2 f, \ldots, \Koop^k f, \ldots$ are meaningful.
It is obvious that $\Koop$ is linear.  

Now, evaluating an observable $f$ along the trajectory of (\ref{eq:DDS}) can be described by evaluating the Krylov sequence $f, \Koop f, \Koop^2 f, \ldots$ at the initial condition $\x_1$.  This remarkably simple idea has far reaching consequences. Although defined using the nonlinear function $\DDS$, $\Koop$ is linear, and studying (\ref{eq:DDS}) by considering observables $f\in\FF$ and how they change under $\Koop$ is linearization of the problem. It seems like a winning trade, but it comes at a cost -- $\Koop$ is defined on the infinite dimensional function space $\FF$ and in practice accessible only through the values $(\Koop^k f)(\x_1)$ for selected observables $f$ and initial conditions $\x_1$. Sometimes multiple trajectories, with different initial conditions, are available. Given the experience from the linear theory, the spectral data of $\Koop$  are of immediate interest. If $\Koop\phi_i\approx\lambda_i\phi_i$, then the approximate eigenfunction $\phi_i\neq \0$ satisfies $\Koop^k\phi_i (\x_1) \approx \lambda_i^k\phi(\x_1)$, i.e. $\phi_i$ is very particular observable for the system (\ref{eq:DDS}) because $\phi_i(\x_2)\approx \lambda_i\phi_i(\x_1)$, $\phi_i(\x_3)\approx \lambda_i^2\phi_i(\x_1)$, $\ldots, \phi_i(\x_{k+1})\approx\lambda_i^k\phi_i(\x_1)$.

The operator $\Koop$ is called the Koopman (composition) operator. It was introduced in 1931 by Bernard Osgood Koopman \cite{Koopman-1931} in the context of classical Hamiltonian mechanics; this work was followed by the first elements of the spectral theory  developed by Koopman and von Neumann \cite{Koopman-vonNeumann-1932}. The composition operator had been studied already by Ernst Schr\"{o}der \cite{Schroeder-composition-1870} in 1870 and Gabriel Xavier Paul Koenigs \cite{Koenighs-compOp-1884} in 1884. In particular, Schr\"{o}der had studied functional equation $f\circ\DDS = \lambda f$, which is the eigenvalue problem $\Koop f=\lambda f$. For a historical review see \cite{Shapiro-CompHistory-1998}.

Suppose that $\ell$ approximate  eigenpairs $(\lambda_i,\phi_i)$ of $\Koop$ are available and that an observable $f$ is successfully well approximated at $\x_1$ as the superposition $f(\x_1)\approx \sum_{i=1}^\ell\alpha_i\phi_i(\x_1)$. Then, for each $k$, $f(\x_k)=(\Koop^{k-1}f)(\x_1)\approx \sum_{i=1}^\ell\alpha_i\lambda_i^{k-1}\phi_i(\x_1)$. This decomposition, called the Koopman Mode Decomposition (\kmd{}), reveals the latent structure of the nonlinear dynamics and allows for forecasting. Hence, finding (at least approximate) eigenpairs of $\Koop$ would be particularly rewarding for computational data driven analysis of the dynamics. For the theory and applications of the Koopman operator to dynamical systems the reader is referred to \cite{Mezic-Spectral-MOR-2005}, \cite{ROWLEY_MEZIC_BAGHERI_SCHLATTER_HENNINGSON_2009},  \cite{mezic_annual_reviews-2013}, \cite{tu-rowley-dmd-theory-appl-2014},   \cite{Mezic-OT-Spectra-2016}, \cite{Mezic-Koop-Spectrum-2020}, \cite{Mezic-Koop-Learning-DS-AMS-2021}, \cite{Colbrook-Drmac-Horning-2025}, and for a review of methods and applications to \cite{Akshay2021}, \cite{brunton-budisic-kaiser-kutz-sirew-2022}.

The eigenpairs of $\Koop$ are computed from its finite dimensional $N\times N$ compression onto suitable $N$--dimensional subspace $\FF_N\subset\FF$ using the Extended Dynamic Mode Decomposition (\edmd{}) \cite{williams-2015-EDMD},  \cite{Williams-etal-Kernel-based-2015}.  The matrix representation of the compression is in a data driven setting computed by solving certain weighted least squares problem that mimics computing $L^2$ projections. \edmd{} builds upon the Dynamic Mode Decomposition (\dmd{}) \cite{Schmid:2008wv}, \cite{Schmid:2010ba} and the matrix representations used in the two methods are transposes of each other. Convergence of \edmd{} is established in  \cite{Korda-Mezic-EDMD-convergence-2018}.

From the computational point of view, there are two main difficulties, even before resorting to numerical methods and software solutions for data driven applications in finite precision computer arithmetic: \\
\emph{(i)} Finite dimensional approximation of eigenpairs of linear operators defined on an infinite dimensional Hilbert or Banach space is difficult task, and there is a risk that spurious eigenvalues might be misidentified as the true eigenvalues (spectral pollution). Important targeted genuine eigenvalues might be missed (spectral invisibility).\\
\emph{(ii)} Once a finite dimensional compression of the operator is constructed, its matrix may be defined by the data only on a subspace (as in the \dmd{} when relatively small number of high dimensional data snapshots is available) so that the Rayleigh--Ritz approximation from a subspace can compute only approximate eigenpairs -- some of them are accurate, while others are spurious.

In numerical linear algebra, approximate eigenpairs are tested by computing the residuals.
The importance of using residual bounds in the Koopman analysis framework  was first  stressed in \cite{Drmac-Mezic-Mohr-EnhancedDMD-2018} in the context of computing the Koopman modes in the \dmd{}, and in \cite{Colbrook-Townsend-Rigor-2024}, \cite{Colbrook_Ayton_Szoke_2023}, \cite{Colbrook-AnotherLook-2024} in the context of eigenpairs of $\Koop$. 
In practice this means that each computed eigenpair $(\lambda_i,\phi_i)$ is tested by computing the norm of the residual $\zbr_i=\Koop\phi_i-\lambda_i\phi_i$, and $\ell\leq N$ pairs with the residual norm below given threshold are selected. More advanced techniques include e.g. probabilistic estimates of the trustworthiness of the computed eigenvalues \cite{Wormell-2026}. 


Despite these extensive computational and approximation studies, the theoretical connection between the data driven finite--dimensional spectral approximations of $\Koop$ and the computed long--term behavior/forecast  of observables along the trajectories of the system (\ref{eq:DDS}) remains incomplete. For, the original problem is not the eigenvalue problem for $\Koop$, but to study observables of (\ref{eq:DDS}). 
This work addresses this gap by establishing a new
framework that connects approximation errors of the $(\lambda_i,\phi_i)$'s (expressed through the residuals $\zbr_i$), the computed \kmd{} and forecast of the observable of interest, and the values of the observable along a nearby true trajectory.

\subsection{The proposed framework}

When it comes to practical application of the Koopman composition operator in the data driven framework, an additional level of difficulty is that the data (that can be sensor readings) is noisy. If the input to an algorithm is noisy, then even the exact computation cannot guarantee that the result is accurate for the intended but inaccessible exact input. The data necessary for the analysis are taken from a trajectory that is hopefully close to the true but inaccessible trajectory. Furthermore, evaluation of an observable along the given approximate trajectory may not be error free. On top of that, a particular numerical algorithm commits unavoidable computation errors (discretization, truncation and rounding errors when computing matrix spectral and the \svd{} decompositions). How accurate is the output and how to justify and use it with confidence in mission critical applications?

In numerical linear algebra, the computed output is often justified by proving that it is the result of error--free computation but with noise, called \emph{backward error}, added to the input. This noise is artificially constructed from the errors of a concrete numerical algorithm using the \emph{backward error analysis} -- if the backward error is comparable in size with the estimated initial uncertainty in the data, then the output of the algorithm is considered as good as one can hope for. 

To be useful and interpretable, the backward error must be expressed in  terms of the given problem. For studying the observables of (\ref{eq:DDS}) using \edmd{}/\kmd{}, we propose a new framework called \emph{backward shadowing analysis} -- coupling of backward error analysis, that addresses computing the approximate eigenpairs $(\lambda_i,\phi_i)$ of $\Koop$, and the \emph{shadowing theory} that pushes the \edmd{} error into the trajectory of (\ref{eq:DDS}). 

\subsubsection{Main results}
The first step in this {backward shadowing analysis} is to show that the selected approximate eigenpairs $(\lambda_i,\phi_i)$, $i=1,\ldots, \ell$, of $\Koop$ are exact for $\Koop-\Delta\Koop$, where $\Delta\Koop$ is small if all residuals $\zbr_i=\Koop\phi_i-\lambda_i\phi_i$ are small. This follows by a direct application of well known techniques from numerical linear algebra. The operator $\Delta\Koop$ is the backward error that makes approximate eigenpairs exact for $\Koop-\Delta\Koop$. While useful for spectral perturbation theory of $\Koop$, this information is not satisfactory in the context of computational analysis of the given dynamical system: $\Koop-\Delta\Koop$ is in general not a composition operator and it cannot be associated with the original dynamics (\ref{eq:DDS}).
The problem of proper interpretation of the computed result escalates if $\Koop$ has additional properties, like unitarity which corresponds to measure preserving systems with important interpretation in terms of the underlying physics of the problem. Hence, using the $(\lambda_i,\phi_i)$'s immediately breaks the connection to $\DDS$ that defines the original problem (\ref{eq:DDS}), which precludes justification of the results in terms of backward error analysis.

We resolve this issue using the following \emph{shadowing trick in two steps}: First, the powers of 
$\Koop-\Delta\Koop$ applied to an observable $f$ and evaluated at $\x_1$ are interpreted as evaluations of that observable along a pseudo--trajectory of (\ref{eq:DDS}). Then, that pseudo--trajectory is shadowed by a true trajectory of (\ref{eq:DDS}). This yields, to the best of our knowledge, the first backward stability result for the application of the \edmd{}/\kmd{} for the Koopman operator based analysis and forecasting of dynamical systems.

Our proposed framework synergizes, in the concrete example of \edmd{}/\kmd{}, the backward error analysis and the shadowing theory -- a connection anticipated by Stuart \cite[\S 7.2]{Stuart-NumAn-DynSys-1994} who considered them important concepts with increasingly important role in computational analysis of dynamical systems.

Finally, as a by--product of the analysis, we also obtain a refinement of the numerical shadowing lemma, that guarantees existence of true shadowing trajectory which intersects the pseudo--trajectory at certain number of indices. 

\subsubsection{Organization of the paper}
The paper focuses on the conceptual foundation of a new framework for the numerical analysis of the \edmd{}/\kmd{}, leaving many technical details and challenging research problems for future work. It is organized as follows. Section \ref{S=Preliminaries} sets the stage -- it introduces the Koopman operator for discrete dynamical systems and explains the details of the data driven compression (\S \ref{SS=Compresionmatrix}) and computation of approximate eigenpairs (\S \ref{SS=AprroxEigenpairs}). Section \ref{SEC:ResidsBackwardpert} explains the construction of the backward error $\Delta\Koop$ for the general case of linear operator $\Koop$ on Hilbert space. 
Section \ref{SS=BackErrShadowOpenPr} specifies $\Koop$ as the composition operator attached to the dynamics (\ref{eq:DDS}) and takes on the problem of proper interpretation of $\Delta\Koop$, because $\Koop-\Delta\Koop$ is not the Koopman operator of (\ref{eq:DDS}). The proposed approach, outlined in \S \ref{SS=ConceptualFramework}, is in the framework of the shadowing theory. 
Some new insights into the numerical shadowing are given in \S \ref{SSS=ShadowCOntactBackStable}. The new proposed concept of backward shadowing stability of the \edmd{} is outlined in \S \ref{S=MixedBackShadowEDMD}. Practical data driven formulas to approximate the residuals and the operator norm of $\Delta\Koop$ are given in \S \ref{S=DataDrivenResids}. Concluding remarks and discussion of future research along the ideas from this work are in \S \ref{S=ConcludingRem}.

\section{Preliminaries}\label{S=Preliminaries}
The ambient space $\FF$ of observables for applications of $\Koop$ 
is separable Hilbert space with the inner product $\langle \cdot,\cdot \rangle$ and the  norm $\|\cdot\|=\sqrt{\langle \cdot,\cdot \rangle}$. The induced operator norm is denoted by $\vertiii{\cdot}$. Natural setting for many applications is $\FF=L^2(\mathcal{X},\mu)$, with forward invariant $\mathcal{X}$ and an appropriate measure $\mu$, or reproducing kernel Hilbert space (RKHS) with judicious choice of kernel.

Finite dimensional approximation of $\Koop$ begins with the selection of 
$N$ dimensional subspace $\FF_N\subset \FF$, which is given as the linear span of suitably selected dictionary (basis) $\mathcal{B}=\{\basef_1, \ldots, \basef_N\}$. 
In theory, with well selected $\mathcal{B}$ and large enough dimension $N$, replacing $\Koop$ with its compression mitigates the problem of infinite dimension. Convergence, in the framework of operator approximation theory, can be analyzed as $N\longrightarrow\infty$, see e.g. \cite{Arveson1992CAlgebrasAN}, \cite{Arveson-1993}, \cite{Chatelin-SALO-2011},  \cite{Hansen-Approx-spectra-2008}.

In data driven scenarios, the problem is more difficult because only a finite number $M$ of data snapshot pairs $(\x_k,\y_k=\DDS(\x_k))$, $k=1,\ldots, M$, is available and all functions/observables are defined only as finite tables of values. For proper convergence analysis, in addition to $N\longrightarrow\infty$, the limit $M\longrightarrow\infty$ must be considered.

\subsection{Data driven compression of $\Koop$ and its matrix representation}\label{SS=Compresionmatrix}
Numerical data driven method for working with the compression of $\Koop$ to $\FF_N$  is an adaptation of the construction of matrix representation of a linear operator. Consider $\Koop \basef_i = \basef_i \circ\DDS$ and  split it as the sum of the component belonging to $\FF_N$, and the residual, 
\begin{equation}\label{eq:Utpi(x)}
	(\Koop \basef_i)(\zbx) = \basef_i(\DDS(\zbx)) = \sum_{j=1}^N \uuk_{ji} \basef_j(\zbx) + \rho_i(\zbx), \;\;i=1,\ldots, N.
\end{equation}
Projecting $\Koop\basef_i$ back to $\mathcal{F}_N$ means finding the coefficients $\uuk_{ji}$ to 
minimize the norm of the residual, i.e. to make the residual orthogonal to $\mathcal{F}_N$. The $\uuk_{ji}$'s define the matrix representation $\KoopM_N$ of the compression.
If the only available data are the snapshot pairs $(\x_k,\y_k=\DDS(\x_k))$, 
then it is only feasible to minimize the norms of the residuals
\begin{equation}\label{eq:resid-rhoixk}
	\rho_i(\x_k) = (\Koop \basef_i)(\x_k) - \sum_{j=1}^N \uuk_{ji} \basef_j(\x_k)= \basef_i(\y_k) - \sum_{j=1}^N \uuk_{ji} \basef_j(\x_k),\;\;k=1,\ldots, M, 
\end{equation}
over the $\x_k$'s, i.e. to minimize 
\begin{eqnarray}
	\!\!\!\!\!\!\!\!\!\!\!\! && \sum_{k=1}^M w_k | \sum_{j=1}^N \uuk_{ji} \basef_j(\x_k) - \basef_i(\y_k)|^2 
	= \nonumber \\
	\!\!\!\!\!\!\!\!\!\!\!\! && =
	\left\| \zbW^{\frac{1}{2}}  \left[ 
	\begin{pmatrix} \basef_1(\x_1)  & \ldots & \basef_N(\x_1) \cr
		\vdots & \ldots & \vdots \cr
		\basef_1(\x_M)  & \ldots & \basef_N(\x_M)\end{pmatrix}
	\begin{pmatrix} \uuk_{1i}\cr \vdots \cr \uuk_{Ni}\end{pmatrix} -  
	\begin{pmatrix} \basef_i(\y_1) \cr
		\vdots  \cr
		\basef_i(\y_M) \end{pmatrix}
	\right]\right\|_2^2 \!\!. \label{eq:LS:Upi}
\end{eqnarray}
In the $L^2$ setting, this is data driven minimization of the residual
$\int_{\mathcal{X}}|\rho(\zbx)|^2d\mu(\zbx)$. The weight matrix $\zbW$ can be determined by a cubature formula and rigorous error bounds are possible, see e.g. 
\cite{Colbrook-Townsend-Rigor-2024}, \cite{Colbrook_Ayton_Szoke_2023}.
If the values of the basis functions on the data are arranged in the matrices $\OX$, $\OY$ as  $(\OX)_{ki}=\basef_i(\x_k)$, $(\OY)_{ki}=\basef_i(\y_k)$, then the solution of this weighted least squares problem can be compactly written as
\begin{eqnarray*}
	\KoopM_N(:,i) &=&	\begin{pmatrix} \uuk_{1i}\cr \vdots \cr \uuk_{Ni}\end{pmatrix} = (\OX^* \zbW \OX)^{-1}\OX^* \zbW \begin{pmatrix} \basef_i(\y_1) \cr
		\vdots  \cr
		\basef_i(\y_M) \end{pmatrix} = [\bfPhi_N \Koop\basef_i]_{\mathcal{B}} \\ 
	&=& (\OX^* \zbW \OX)^{-1}\OX^* \zbW \begin{pmatrix} \basef_i(\y_1) \cr
		\vdots  \cr
		\basef_i(\y_M) \end{pmatrix} = (\OX^* \zbW \OX)^{-1}\OX^* \zbW \OY(:,i) 
	= [\bfPhi_N \Koop_{|\mathcal{F}_N}]_{\mathcal{B}}(:,i),
\end{eqnarray*}
where $\bfPhi_N$ is the projection in the sense of the (algebraic) least squares, and $[\cdot]_{\mathcal{B}}$ denotes representation of vectors/operators in the basis $\mathcal{B}$.
Each residual $\rho_i$ is in the discrete inner product orthogonal to the entire $\mathcal{F}_{N}$, i.e. 
\begin{equation}\label{eq:rhoi-orthog}
	\sum_{k=1}^M w_k \rho_i(\x_k) \overline{\basef_j(\x_k)} = 0,\;\;j=1,\ldots, N.
\end{equation}
Then, for any $f=\sum_{i=1}^N \mathsf{f}_i \basef_i\in\FF_N$, 
\begin{eqnarray*}
	g(\zbx)\!\! &=& \!\!(\Koop f)(\zbx) = \sum_{i=1}^N \mathsf{f}_i \left[ \sum_{j=1}^N \uuk_{ji} \basef_j(\zbx) + \rho_i(\x) \right] = \sum_{j=1}^N \basef_j(\x) \sum_{i=1}^N \uuk_{ji}\mathsf{f}_i + \sum_{i=1}^N \mathsf{f}_i\rho_i(\zbx) \\
	&=&\!\! \sum_{j=1}^N \basef_j(\zbx) \mathsf{g}_j + \sum_{i=1}^N \mathsf{f}_i\rho_i(\zbx),\;\;\mbox{where}\;\; \mathsf{g}_j=\sum_{i=1}^N \uuk_{ji}\mathsf{f}_i .
\end{eqnarray*}
It should be noted that $\sum_{i=1}^N \mathsf{f}_i\rho_i(\zbx)$ is orthogonal to $\mathcal{F}_{N}$ in the discrete inner product; see (\ref{eq:rhoi-orthog}). This confirms $\sum_{j=1}^N \basef_j(\zbx) \mathsf{g}_j$ as the projection (in the discrete sense of the algebraic least squares) of $\Koop f(\zbx)$ onto $\mathcal{F}_N$. To summarize:

\begin{proposition}\label{PROP-KFN}
	The data driven operator compression $\bfPhi_N \Koop_{|\mathcal{F}_N} : \mathcal{F}_N \longrightarrow \mathcal{F}_N$ is in the basis $\mathcal{B}$ represented by the matrix
	\begin{eqnarray}\label{eq:UN}
		\!\!\! &&[\bfPhi_N \Koop_{|\mathcal{F}_N}]_{\mathcal{B}} =  (\OX^* \zbW \OX)^{-1}\OX^* \zbW \OY \equiv \KoopM_N ,\\
		\!\!\! && \bfPhi_N \Koop_{|\mathcal{F}_N} \begin{pmatrix} \basef_1(\zbx) & \ldots & \basef_N(\zbx)\end{pmatrix}\begin{pmatrix}\mathsf{f}_1 \cr \vdots\cr \mathsf{f}_N\end{pmatrix} = \begin{pmatrix} \basef_1(\zbx) & \ldots & \basef_N(\zbx)\end{pmatrix} (\KoopM_N \begin{pmatrix}\mathsf{f}_1 \cr \vdots\cr \mathsf{f}_N\end{pmatrix}) ,\nonumber
	\end{eqnarray}
	i.e. the operator $\bfPhi_N \Koop_{|\mathcal{F}_N}$ is in the basis $\mathcal{B}$ 
	represented by the matrix as a linear operator on $\C^N$ defined by 
	$$
	\begin{pmatrix}\mathsf{f}_1 \cr \vdots\cr \mathsf{f}_N\end{pmatrix} \longrightarrow 
	\begin{pmatrix}\mathsf{g}_1 \cr \vdots\cr \mathsf{g}_N\end{pmatrix} = 
	\KoopM_N \begin{pmatrix}\mathsf{f}_1 \cr \vdots\cr \mathsf{f}_N\end{pmatrix} .
	$$
	In the simpler unweighted case with $\zbW=\Id_M$, $\KoopM_N = (\OX^* \OX)^{-1} \OX^* \OY = \OX^\dagger \OY$.
	Further, the action of $\Koop$ on $f=\sum_{i=1}^N \mathsf{f}_i \basef_i\in\FF_N$ is
	\begin{equation}\label{eq:Kfplusrho}
		(\Koop f)(\zbx) = \begin{pmatrix} \basef_1(\zbx) & \ldots & \basef_N(\zbx)\end{pmatrix} (\KoopM_N \begin{pmatrix}\mathsf{f}_1 \cr \vdots\cr \mathsf{f}_N\end{pmatrix}) + \begin{pmatrix}
			\rho_1(\zbx) & \ldots & \rho_N(\zbx)
		\end{pmatrix}\begin{pmatrix}\mathsf{f}_1 \cr \vdots\cr \mathsf{f}_N\end{pmatrix}.
	\end{equation}
\end{proposition}
\noindent For convergence results in the case of bounded $\Koop$ see \cite{Korda-Mezic-EDMD-convergence-2018}, \cite{Korda-Putinar-Mezic-2020}. In \cite{Llamazares-Elias-convergence-2026}, $\Koop$ is assumed closed and then the underlying Hilbert space structure on the operator domain $\mathcal{D}(\Koop)$ is changed by using the graph inner product $\langle f,g\rangle_{\mathfrak{g}}=\langle f,g\rangle + \langle \Koop f, \Koop g\rangle$ and the corresponding norm $\|f\|_{\mathfrak{g}}^2 = \|f\|^2+\|\Koop f\|^2$. Then $\Koop: (\mathcal{D}(\Koop),\langle\cdot,\cdot,\rangle_{\mathfrak{g}})\longrightarrow (\FF,\langle\cdot,\cdot\rangle)$ is bounded, $\|\Koop f\| \leq \|f\|_{\mathfrak{g}}$ for all $f\in \mathcal{D}(\Koop)$, and the corresponding induced norm is $\vertiiis{\Koop}_{\mathfrak{g}}=\sup_{f\neq\0}\|\Koop f\|/\|f\|_{\mathfrak{g}} \leq 1$.

In the sequel, to minimize technical details, the tacit assumption is that $M\gg N$ and that $\OX$ is of full numerical rank $N$. (In case of numerical rank $r<N$ of $\OX$, the matrix representation of the corresponding compression of $\Koop$ is $r\times r$ $\KoopM_r$, which corresponds to using an $r$--dimensional subspace of $\FF_N$.)

\subsection{Approximate eigenpairs}\label{SS=AprroxEigenpairs}
In the generic case, $\KoopM_N$ is diagonalizable, i.e.
$\KoopM_N = \bfSS \Lambda \bfSS^{-1}$, with $\Lambda=\mathrm{diag}(\lambda_i)_{i=1}^N$, $\bfSS=(\bfss_1,\ldots,\bfss_N)$. The columns of $\bfSS$ are the eigenvectors, $\KoopM_N \bfss_i=\lambda_i \bfss_i$, and they are normalized so that $\|\bfss_i\|_2=1$. Then, based on the spectral decomposition of $\KoopM_N$, the eigenvalues and the eigenfunctions \index{eigenfunction} of $\Koop$ can be approximated from the linear span of $\basef_1,\ldots, \basef_N$ as follows.

Inserting the spectral decomposition of $\KoopM_N$ in the formula for the action of $\Koop$ on the dictionary, together with (\ref{eq:Utpi(x)}), 
gives\footnote{This notation can be interpreted using
	$\infty \times N$ quasi matrices \cite{Trefethen-HQR-quasim-2010}, \cite{Townsend-Trefethen-Cont-Matr-2015}.}
\begin{equation}\label{eq:UfV}
	\begin{pmatrix}
		\Koop\basef_1(\zbx) & \ldots & \Koop\basef_N(\zbx)
	\end{pmatrix}  =  \begin{pmatrix}
		\basef_1(\zbx) & \ldots & \basef_N(\zbx)
	\end{pmatrix} \bfSS\Lambda \bfSS^{-1} + \begin{pmatrix}
		\rho_1(\zbx) & \ldots & \rho_N(\zbx)
	\end{pmatrix},
\end{equation}
and multiplying this from the right with $\bfSS$ gives (using the linearity of $\Koop$)
\begin{eqnarray}\label{eq:UfV-2}
	\!\!\!\!\!\!\!	&& \begin{pmatrix}
		\Koop (\sum_{j=1}^N \bfSS_{j1}\basef_j(\zbx)) & \ldots & \Koop(\sum_{j=1}^N \bfSS_{jN}\basef_j(\zbx))
	\end{pmatrix}  \\ 
	\!\!\!\!\!\!\!	&& = \!  \begin{pmatrix}
		\sum_{j=1}^N\! \bfSS_{j1}\basef_j(\zbx) & \ldots & \sum_{j=1}^N\! \bfSS_{jN}\basef_j(\zbx)
	\end{pmatrix}\!\! \begin{pmatrix} \lambda_1 & & \cr 
		& \ddots & \cr & & \lambda_N\end{pmatrix} + \begin{pmatrix}
		\rho_1(\zbx) & \ldots & \rho_N(\zbx)
	\end{pmatrix}\!\bfSS .\nonumber
\end{eqnarray}
(This relation follows by direct application of the second part of Proposition \ref{PROP-KFN}.)
Hence, the functions $\phi_1, \ldots, \phi_N$ defined by 
\begin{equation}\label{eq-eigfun}
	\begin{pmatrix}
		\phi_1(\zbx)  & \ldots &  \phi_N(\zbx)
	\end{pmatrix} =  \begin{pmatrix}
		\basef_1(\zbx) & \ldots &  \basef_N(\zbx)
	\end{pmatrix}\bfSS ,\;\mbox{i.e.}\; \phi_i = \sum_{j=1}^N \bfSS_{ji}\basef_j, 
\end{equation}
can be taken as approximate eigenfunctions of $\Koop$, because
\begin{equation}\label{eq-eigfun-2}
	(\Koop \phi_i)(\zbx) = \lambda_i \phi_i(\zbx) + \zbr_i(\zbx),\;\;\;\zbr_i(\zbx)=\sum_{j=1}^N \rho_j(\zbx) \bfSS_{ji}.
\end{equation}
\noindent In a data--driven application, the observables are given as evaluated at the states $\x_1,\ldots, \x_M$ (the states themselves might be unknown) and the approximate eigenfunctions exist only as tabulated values for $\zbx\in\{\x_1,\ldots, \x_M\}$:
\begin{equation}\label{eq:eigs-tabulated}
	\bfPhi_x =	\begin{pmatrix}
		\phi_1(\x_1) &   \ldots & \phi_N(\x_1) \cr
		\phi_1(\x_2) &   \ldots & \phi_N(\x_2) \cr
		\vdots & \vdots  &  \vdots \cr
		\phi_1(\x_{M})  & \ldots & \phi_N(\x_{M}) \cr
	\end{pmatrix}
	= 
	\begin{pmatrix}
		\basef_1(\x_1)  &  \ldots & \basef_N(\x_1) \cr
		\basef_1(\x_2)  &  \ldots & \basef_N(\x_2) \cr
		\vdots & \vdots  &  \vdots \cr
		\basef_1(\x_{M})  &  \ldots & \basef_N(\x_{M}) \cr
	\end{pmatrix}\bfSS .
\end{equation}
Since $\basef_1,\ldots,\basef_N$ are assumed linearly independent, with $\rank(\OX)=N < M$, and $\bfSS$ is assumed nonsingular, $\phi_1,\ldots,\phi_N$ are linearly independent and the matrix of their tabulated values at $\x_1,\ldots,\x_M$ in (\ref{eq:eigs-tabulated}) is of full column rank.
\begin{remark}
	{\em
		In the non--normal case, an alternative to the classical \edmd{} is the Koopman--Schur decomposition \cite{Drmac-Mezic-Koopman-Schur-2026}, which instead of eigenfunctions computes orthonormal bases for nested sequence of invariant subspaces.
	}
\end{remark}
\section{Residuals and  backward perturbations of $\Koop$}\label{SEC:ResidsBackwardpert}
In this section $\Koop : \mathcal{D}(\Koop)\subseteq\FF \longrightarrow \FF$ is  densely defined closed operator, which includes the bounded case, e.g. unitary composition operators for measure preserving systems or quasi--compact operators.  The resolvent set and the spectrum of $\Koop$ are denoted by $\resolvent(\Koop)$ and $\bfsigma(\Koop)$ respectively. The point spectrum (eigenvalues) is denoted by $\bfsigma_p(\Koop)$.

Theoretically, the quality of an eigenpair $(\lambda_i,\phi_i)$  can be estimated from the residual $\zbr_i = \Koop\phi_i - \lambda_i\phi_i$. 
Note that the residual measures how well is $\Koop\phi_i$ approximated from the linear span of $\phi_i$. This approximation is optimal if $\zbr_i$ is orthogonal to $\phi_i$. Hence, once  the pair $(\lambda_i,\phi_i)$ is computed by an algorithm, it makes sense to redefine $\lambda_i$ to the Rayleigh quotient $\lambda_i=\langle\Koop\phi_i,\phi_i\rangle/\langle \phi_i, \phi_i\rangle$ and ensure $\langle \zbr_i,\phi_i\rangle =0$. 

Small residual norm $\|\zbr_i\|$ does not automatically mean that $\lambda_i$ is close to some eigenvalue of $\Koop$, in particular when $\Koop$ is nonnormal. Even in the case of finite nonnormal matrices, small residuals may be quite misleading, see \cite{Kahan-Parlett-Jiang-1982}.
The distance from $\lambda_i$ to the closest eigenvalue of $\Koop$ is a separate question answered by a combination of backward error analysis and perturbation theory.  Nevertheless, residual is useful because it is computable from the data at hand ($\Koop, \lambda_i, \phi_i$) and can be used to interpret the computed eigenpair as exact for an engineered \emph{backward} perturbation $\Koop-\Delta\Koop$ of the operator $\Koop$. The details of the technique of turning the residuals of a selection of approximate eigenpairs into a single backward  error $\Delta\Koop$ are given in \S \ref{SS=IndivResidAggregBackErr} in a general setting, where $\Koop$ is not necessarily the Koopman composition operator. The low--rank structure of $\Delta\Koop$ and its operator norm $\vertiii{\Delta\Koop}$ are given in \S \ref{SS=LowRankFactNormDK}. An auxiliary technical details needed for the construction of $\Delta\Koop$ is explained in \S \ref{SSS=tildephis}.

\subsection{From individual residuals to composite backward error in $\Koop$}\label{SS=IndivResidAggregBackErr}
The construction of backward error from the residual is standard technique in numerical linear algebra. In the case of single eigenpair, 
it suffices to set 
\begin{equation}\label{eq;Delta_iKoop}
	\Delta_i\Koop = \frac{1}{\langle \phi_i,\phi_i\rangle} \zbr_i\otimes \phi_i\;\;\;(\;\Delta_i\Koop f = \frac{1}{\langle \phi_i,\phi_i\rangle} \zbr_i\otimes \phi_i=\frac{\langle f,\phi_i\rangle}{\langle \phi_i,\phi_i\rangle} \zbr_i, \;f\in\FF\; ),
\end{equation}
and it is a simple matter to confirm that the computed $(\lambda_i,\phi_i)$ is an exact eigenpair of $\widetilde\Koop=\Koop-\Delta_i\Koop$:
\begin{equation}
	(\Koop - \Delta_i\Koop) \phi_i = \lambda_i\phi_i + \zbr_i - \frac{\langle \phi_i,\phi_i\rangle}{\langle \phi_i,\phi_i\rangle} \zbr_i = \lambda_i\phi_i .
\end{equation}
The norm of the perturbation $\Delta_i\Koop$ is
$$
\vertiiis{\Delta_i\Koop}=\sup_{\|f\|=1}\|\Delta_i\Koop f\|  = 
\sup_{\|f\|=1} \frac{|\langle f,\phi_i\rangle|}{\langle \phi_i,\phi_i\rangle} \|\zbr_i\| = \frac{\|\zbr_i\|}{\|\phi_i\|}.
$$
It is convenient to normalize the eigenfunctions (replace $\phi_i$ with $\phi_i/\|\phi_i\|$) and in the sequel we assume $\|\phi_i\|=1$ for all $i$, and write $\Delta_i\Koop=\zbr_i\otimes\phi_i$, $\vertiiis{\Delta_i\Koop}=\|\phi_i\|\|\zbr_i\|=\|\zbr_i\|$. Hence, the residuals can be used not only to select good approximate eigenpairs, but also to interpret the
computed quantities in the sense of backward error.

It should be  noted that the backward perturbation $\Delta_i\Koop$ in (\ref{eq;Delta_iKoop}) is constructed for each pair $(\lambda_i,\phi_i)$ separately and it yields $\lambda_i\in\bfsigma_p(\Koop-\Delta_i\Koop)$ for $i=1,\ldots, N$. This is not satisfactory because the computed eigenpairs 
(all $N$ or a subset of $\ell<N$ best ones, selected using the residuals) are used as co--players in modal analysis and these separate backward errors for each used eigenpair cannot provide interpretation in the sense of backward stability, i.e. that the data is analyzed in terms of spectral elements of some $\Koop+\mathcal{E}$ with small $\vertiii{\mathcal{E}}$.  
Similarly, if one considers the pseudospectrum\footnote{See e.g.  \cite{Trefethen-Pseudospectra-SIREW-1997},  \cite{Trefethen-Grad-NA-1999},  \cite{Trefethen-Embree-Book-2005}} of $\Koop$,
\begin{equation}
	\bfsigma_{\varepsilon}(\Koop)=\bfsigma(\Koop)\bigcup \{\zeta\not\in\bfsigma(\Koop), \vertiiis{(\zeta \Id - \Koop)^{-1}}\geq \frac{1}{\varepsilon}\}	= \overline{\bigcup_{\vertiii{\mathcal{E}}\leq\varepsilon}\bfsigma(\Koop+\mathcal{E})},
\end{equation}
then it is desirable and more natural to interpret $\lambda_1,\ldots, \lambda_N$ as the eigenvalues of single operator $\Koop+\mathcal{E}$, rather than $\lambda_i\in\bfsigma(\Koop+\mathcal{E}_i)$, $i=1,\ldots, N$.

The following theorem shows how to construct a composite backward error $\Delta\Koop$ for any selection of $\ell$ eigenpairs\footnote{To ease the notation, assume indexing such that the leading $\ell\leq N$ pairs are selected. For instance, the eigenpairs can be indexed by listing the ones with smallest residuals first.} $(\lambda_i,\phi_i)$, $i=1,\ldots, \ell$, that can be then interpreted as eigenvalues of some $\Koop-\Delta\Koop$.

\begin{theorem}\label{TM:DeltaKoop}
	Let $(\lambda_i,\phi_i)$ be approximate eigenpairs with $\|\phi_i\|=1$, $i=1,\ldots, \ell$, and let $\{\widetilde\phi_1, \ldots, \widetilde\phi_\ell\}$ and $\{\phi_1,\ldots, \phi_\ell\}$ be biorthogonal, $\langle \phi_i, \widetilde\phi_j\rangle =\bfdelta_{ij}$.\footnote{Here $\bfdelta_{ij}$ denotes the Kronecker's $\bfdelta$: $\bfdelta_{ij}=1$ for $i=j$; otherwise $\bfdelta_{ij}=0$.}	
	Let 
	\begin{equation}
		\Delta\Koop = \sum_{j=1}^\ell \zbr_j\otimes\widetilde\phi_j,\;\;\zbr_j=\Koop\phi_j-\lambda_j\phi_j.
	\end{equation}
	Then $\{\lambda_1,\ldots,\lambda_\ell\}\subset \bfsigma_p(\Koop-\Delta\Koop)$, $\vertiiis{\zbr_j\otimes\widetilde\phi_j}=\|\zbr_j\|\|\widetilde\phi_j\|$, $j=1,\ldots, \ell$.
	The essential spectra of $\Koop$ and $\Koop-\Delta\Koop$ are identical (since $\Delta\Koop$ is of finite rank).
\end{theorem}
{\sc Proof} It is easily checked that
	\begin{eqnarray*}
		(\Koop - \sum_{j=1}^\ell \zbr_j\otimes\widetilde\phi_j)\phi_i &=& \lambda_i\phi_i + \zbr_i - 
		\sum_{j=1}^\ell (\zbr_j\otimes\widetilde\phi_j) \phi_i \\ &=& 
		\lambda_i\phi_i + \zbr_i - \sum_{j=1}^\ell \langle \phi_i, \widetilde\phi_j\rangle \zbr_j = 
		\lambda_i\phi_i + \zbr_i - \zbr_i = \lambda_i\phi_i .
	\end{eqnarray*}
	(Note that here $\Koop$ is not necessarily the composition operator.)
\hfill$\boxtimes$

\begin{remark}
	{\em
		If the $\phi_i$'s are not normalized, then $\Delta\Koop$ is written  as 
		$
		\Delta\Koop = \sum_{j=1}^\ell \frac{1}{\langle \phi_j,\phi_j\rangle}\zbr_j\otimes\widetilde\phi_j,
		$
		and $\widetilde\phi_1, \ldots, \widetilde\phi_\ell$ are defined by the biorthogonality relations $\langle \phi_i,\widetilde\phi_j\rangle=\bfdelta_{ij}\langle\phi_i,\phi_i\rangle$, $1\leq i,j\leq \ell$. If $\phi_1,\ldots, \phi_\ell$ are orthonormal, then $\widetilde\phi_i=\phi_i$.
	}
\end{remark}

\noindent  Theorem \ref{TM:DeltaKoop} allows for accepting the computed eigenpairs as exact, provided that small perturbation of the operator is acceptable. If all residuals are sufficiently small, then $\vertiii{\Delta\Koop}\leq \sum_{j=1}^\ell \|\zbr_j\|\|\widetilde\phi_j\|$ is small. That  the computed eigenpairs are exact for $\Koop-\Delta\Koop$ with small $\vertiii{\Delta\Koop}$ and $\Delta\Koop$ of rank at most $\ell$ is useful information.  Low (finite) rank perturbations can create or destroy eigenvalues, but leave the essential spectrum intact. Depending on $\Koop$, perturbation theory can be applied to estimate the distances between the $\lambda_i$'s and the spectrum of $\Koop$. This is further pursued in \cite{Drmac-Djimoti-BackShadow-2026}.

		%

\subsection{Low rank factorization and the norm of $\Delta\Koop$}\label{SS=LowRankFactNormDK}
The formula for $\Delta\Koop$ can be written in the familiar rank--revealing product form as follows. 
Define the symbol $\zbR$ as the $\infty \times \ell$ matrix\footnote{See \cite{Townsend-Trefethen-Cont-Matr-2015}.} whose columns are the
functions $\zbr_1,\ldots, \zbr_{\ell}$, and interpret it as the linear operator
\begin{equation}\label{eq:Resid-operator}
	\zbR=\begin{pmatrix} \zbr_1 & \ldots & \zbr_\ell\end{pmatrix} : \C^\ell\longrightarrow \FF, \;\; \zbR \zbz=\sum_{i=1}^\ell z_i \zbr_i,\;\;\zbz=\begin{pmatrix} z_1 & \ldots & z_\ell\end{pmatrix}^T.
\end{equation}
In the same way, define $\Phi\zbz=\sum_{i=1}^\ell z_i\phi_i$ and 
\begin{equation}\label{eq:biorthPhi}
	\widetilde\Phi = \begin{pmatrix} \widetilde\phi_1 & \ldots & \widetilde\phi_\ell\end{pmatrix}: \C^\ell\longrightarrow \FF, \;\; \widetilde\Phi \zbz=\sum_{i=1}^\ell z_i \widetilde\phi_i,\;\;\zbz=\begin{pmatrix} z_1 & \ldots & z_\ell\end{pmatrix}^T .
\end{equation}
Then, in the pair of the standard inner product $(\cdot,\cdot)$ in $\C^\ell$ and $\langle\cdot,\cdot\rangle$ in $\FF$, for any $\zbz=\begin{pmatrix} z_1 & \ldots , z_\ell\end{pmatrix}^T\in\C^\ell$ and $f\in\FF$
$$
\langle \widetilde\Phi \zbz, f\rangle = \langle \sum_{i=1}^\ell z_i \widetilde\phi_i, f\rangle = \sum_{i=1}^\ell z_i \langle \widetilde\phi_i, f\rangle = (\zbz,\begin{pmatrix} \overline{\langle \widetilde\phi_1, f\rangle} & \ldots & \overline{\langle \widetilde\phi_\ell, f\rangle}\end{pmatrix}^T) ,
$$
and the adjoint $\widetilde\Phi^*$ of $\widetilde\Phi$ that satisfies $\langle \widetilde\Phi \zbz, f\rangle = (\zbz,\widetilde\Phi^* f)$ is 
$$
\widetilde\Phi^* f = \begin{pmatrix} \langle f,\widetilde\phi_1\rangle \cr \vdots \cr \langle f,\widetilde\phi_\ell\rangle\end{pmatrix}\in\C^\ell,\;\;f\in\FF .
$$
The formulas for the adjoints $\zbR^*$, $\Phi^*$ are analogous. 
\begin{proposition}\label{PROP:DKAdjoint}
	With $\zbR$ and $\widetilde\Phi$ defined in (\ref{eq:Resid-operator}), (\ref{eq:biorthPhi}), the backward perturbation $\Delta\Koop$ can be written analogously to the matrix low--rank factorization as $\Delta\Koop=\zbR\widetilde\Phi^*$. The adjoint $\Delta\Koop^{\star}$ of $\Delta\Koop$ in $(\FF,\langle\cdot,\cdot\rangle)$ is 
	$\Delta\Koop^\star = \widetilde\Phi\zbR^*$.
\end{proposition}
{\sc Proof} It is straightforward to compute that, for any $f, g\in\FF$, 
	\begin{equation}
		(\Delta\Koop f) (\zbx) = \sum_{j=1}^\ell ((\zbr_j\otimes\widetilde\phi_j) f)(\zbx) = (\zbR(\widetilde\Phi^* f))(\zbx) = \sum_{j=1}^\ell \langle f, \widetilde\phi_j\rangle \zbr_j(\zbx),
	\end{equation}
	and
	$
	\langle \Delta\Koop f,g\rangle = \langle \zbR\widetilde\Phi^* f,g\rangle = (\widetilde\Phi^* f,\zbR^* g)=\langle f,\widetilde\Phi\zbR^*g\rangle = \langle f,(\zbR\widetilde\Phi^*)^\star g\rangle
	$. Hence 
	$\Delta\Koop^\star g = \sum_{j=1}^\ell \langle g,\zbr_j\rangle \widetilde\phi_j$.
\hfill$\boxtimes$

\noindent The next task is to compute the norm of $\Delta\Koop$. 

\begin{proposition}\label{PROP:DeltaKnorm}
	Let $\Delta\Koop$ be defined as in Theorem \ref{TM:DeltaKoop}. Then 
	\begin{equation}\label{eq:|||DK|||}
		\vertiii{\Delta\Koop} = \sqrt{\|\zbL^* \zbM_{\zbR}\zbL\|_2} = \sqrt{\lambda_{\max}(\widetilde\GramM \zbM_{\zbR})}= \sigma_{\max}(\sqrt{\zbM_{\zbR}}\zbL),
	\end{equation}
	where $\zbM_{\zbR}$ is the Gram matrix of the residuals $\zbr_1,\ldots,\zbr_{\ell}$, $\sqrt{\zbM_{\zbR}}$, its positive semidefinite square root; $\zbL$ is the lower triangular Cholesky factor of the Gram matrix $\widetilde\GramM$ of $\widetilde\phi_1,\ldots, \widetilde\phi_\ell$, $\|\cdot\|_2$ is the matrix spectral norm, and $\sigma_{\max}(\cdot)$ denotes the maximal singular value. (If $\widetilde\phi_1,\ldots, \widetilde\phi_\ell$ are orthonormal, then $\zbL=\Id_\ell$.)
\end{proposition}
{\sc Proof} Since $\Delta\Koop^\star\Delta\Koop$ is Hermitian positive semidefinite with finite dimensional range,
	$$
	\vertiii{\Delta\Koop}^2=\vertiii{\Delta\Koop^\star\Delta\Koop}=\sup_{\|f\|=1}\langle \Delta\Koop^\star\Delta\Koop f,f\rangle = \kappa_{\max} ,
	$$
	where $\kappa_{\max}$ is the largest nonnegative eigenvalue of $\Delta\Koop^\star\Delta\Koop=\widetilde\Phi\zbR^*\zbR\widetilde\Phi^*$. The key observation is that $\Delta\Koop^\star\Delta\Koop$ and $\widetilde\Phi^*\widetilde\Phi\zbR^*\zbR :\C^\ell\longrightarrow\C^\ell$ have the same set of positive eigenvalues. This follows from a general result and it can be shown as follows. Let $\kappa>0$ be an eigenvalue of $\Delta\Koop^\star\Delta\Koop$: 
	$\widetilde\Phi\zbR^*\zbR\widetilde\Phi^* \psi = \kappa\psi$, $\psi\neq\0$. Then necessarily $\widetilde\Phi^* \psi\neq \0$ and 
	$\widetilde\Phi^*\widetilde\Phi\zbR^*\zbR\Phi^* \psi = \kappa \Phi^* \psi$, i.e. $\kappa$ is an eigenvalue of $\widetilde\Phi^*\widetilde\Phi\zbR^*\zbR$. A similar argument gives the other inclusion. The eigenvalues of $\widetilde\Phi^*\widetilde\Phi\zbR^*\zbR$ can be computed from its matrix representation in the canonical basis $\ei_1,\ldots, \ei_\ell$ of $\C^\ell$, which is the product $\zbM_{\widetilde\Phi}\zbM_{\zbR}$ of the matrix representations $\zbM_{\zbR}$ of $\zbR^*\zbR$ and $\zbM_{\widetilde\Phi}$ of  $\widetilde\Phi^*\widetilde\Phi$, that are computed as
	$$
	(\zbM_{\zbR})_{ij} = (\zbR^*\zbR \ei_j,\ei_i) = \langle \zbR\ei_j,\zbR\ei_i\rangle = \langle \zbr_j,\zbr_i\rangle,\;
	(\zbM_{\widetilde\Phi})_{ij} = \langle \widetilde\phi_i,\widetilde\phi_j\rangle,\;\;1\leq i, j \leq \ell .
	$$
	Since $\widetilde\phi_1,\ldots,\widetilde\phi_\ell$ are linearly independent, $\zbM_{\widetilde\Phi}$ is positive definite  and it has Cholesky factorization $\zbM_{\widetilde\Phi}=\zbL\zbL^*$. Since 
	$\zbM_{\widetilde\Phi}\zbM_{\zbR}=\zbL\zbL^*\zbM$ is similar to $\zbL^*\zbM\zbL$, it follows that 
	$\vertiii{\Delta\Koop^\star\Delta\Koop}=\|\zbL^* \zbM_{\zbR}\zbL\|_2$. A practical formula to approximate (\ref{eq:|||DK|||}) from the data is given in \S \ref{S=DataDrivenResids}.
\hfill$\boxtimes$

\subsection{Construction of $\{\widetilde\phi_1, \ldots, \widetilde\phi_\ell\}$}\label{SSS=tildephis}
The existence of the system $\{\widetilde\phi_1, \ldots, \widetilde\phi_\ell\}$ is assured by the theory, in particular by the Hahn--Banach/Riesz representation theorem. However, it is instructive (and it could be useful for practical purposes) to see the details of explicit constructions of such a system in terms of the dictionary $\Basef=\begin{pmatrix} \basef_1, \basef_2, \ldots, \basef_N\end{pmatrix}$ that defines $\FF_N$ and in terms of the approximate eigenfunctions.

Let $\phi_i(x) = \Basef(x) \bfss_i = \sum_{j=1}^N\basef_j(x) (\bfss_i)_j$. 
The functions $\widetilde\phi_i$ from the biorthogonal set are constructed with the Ansatz $\widetilde\phi_i(x)=\Basef(x) \widetilde\bfss_i = \sum_{j=1}^N\basef_j(x) (\widetilde\bfss_i)_j$, and the task is to determine the vectors $\widetilde\bfss_1, \ldots , \widetilde\bfss_{\ell}$. The biorthogonality conditions read
\begin{eqnarray*}
	\langle \phi_i, \widetilde \phi_j\rangle &=& \langle \sum_{k=1}^N\basef_k (\bfss_i)_k, \sum_{s=1}^N\basef_s (\widetilde\bfss_j)_s\rangle = \sum_{k=1}^N \sum_{s=1}^N (\bfss_i)_k \overline{(\widetilde\bfss_j)_s} \langle \basef_k,\basef_s\rangle \\
	&=& \widetilde\bfss_j^* \GramM \bfss_i = \bfdelta_{ij},\;\;\mbox{for all $i,j\in\{1,\ldots,\ell\}$, where}\;\;\GramM_{sk}= \langle \basef_k,\basef_s\rangle.
\end{eqnarray*}
Let $\bfSS_\ell=\begin{pmatrix} \bfss_1 & \ldots &\bfss_\ell\end{pmatrix}$, $\widetilde\bfSS_\ell=\begin{pmatrix} \widetilde\bfss_1 & \ldots & \widetilde\bfss_\ell\end{pmatrix}$. The biorthogonality condition, recorded in the basis $\Basef$ reads $\widetilde\bfSS_\ell^*\GramM\bfSS_\ell=\bfSS_\ell^*\GramM\widetilde\bfSS_\ell=\Id_\ell$. If $\ell=N$, then $\widetilde\bfSS_N=\GramM^{-1}\bfSS_N^{-*}$ is the unique choice. In the case $\ell<N$, $\widetilde\bfSS_\ell=\GramM^{-1}\bfSS_\ell^{*\dagger}$ is the matrix in the solution manifold  that is of minimal Frobenius norm. 

In the case $\ell<N$, more economical construction and favorable for some technical reasons is from the linear span of $\phi_1,\ldots, \phi_\ell$, i.e. $\widetilde\phi_i=\sum_{k=1}^\ell \phi_k(\widehat\bfss_i)_k$. Repeating the above construction yields the $\ell\times\ell$ matrix $\widehat\bfSS_\ell=\begin{pmatrix} \widehat\bfss_1 & \ldots & \widehat\bfss_\ell\end{pmatrix}=\GramM_\phi^{-1}$, where 
$\GramM_\phi$ is the Gram matrix of $\phi_1,\ldots, \phi_\ell$. Note that 
$\|\widetilde\bfSS_\ell\|_F\leq \|\widehat\bfSS_\ell\|_F$.

With a selection of $\widetilde\bfSS_\ell$ ($\widehat\bfSS_\ell$), the values of the functions $\widetilde\phi_i$ at the data $\x_1, \ldots, \x_M$ can be tabulated (analogously to (\ref{eq:eigs-tabulated})) in the matrix $\widetilde \Phi_x = \OX\widetilde\bfSS_\ell$ ($\widetilde \Phi_x = \OX\widehat\bfSS_\ell$).

\section{From backward error in $\Koop$ to shadowing}\label{SS=BackErrShadowOpenPr}
In the context of linearization and data driven analysis of nonlinear dynamics (\ref{eq:DDS}), the Koopman operator $\Koop$ is not the input -- it is a tool. Theorem \ref{TM:DeltaKoop} addresses the general situation of spectral approximation of a linear operator -- when $\Koop$ is specified as the composition operator, it does not assert that $\Koop-\Delta\Koop$ is a composition operator. Hence, this trade of inexactness of the computed eigenpairs $(\lambda_i,\phi_i)$, $i=1,\ldots,\ell$, for the backward error $\Delta\Koop$ in $\Koop$ does not provide satisfactory backward stability statement in terms of the original problem -- study of observables along trajectories of (\ref{eq:DDS}). The result we aim to establish is illustrated in the diagram in Figure \ref{FIG:Goal-BSS}. 
\begin{figure}[H]
	\[
	\begin{tikzcd}[
		arrow style=tikz,
		>=Latex,
		line width=1.1pt,
		row sep=3em,
		column sep=5.1em
		]
		\begin{array}{l}\x_{k+1}=\DDS(\x_k) \cr
			\x_1, \mbox{ observable } f\end{array}
		\arrow[r,draw=red,"\textcolor{red}{\mathrm{\edmd{}, \kmd{}}}"]
		\arrow[d,draw=blue,swap,"{\textcolor{blue}{\widehat\x_1=\x_1+\Delta \x_1}}"]
		&
		\textcolor{red}{\begin{array}{l} (\lambda_i,\phi_i), i=1,\ldots, \ell ; \; \Koop-\Delta\Koop ;\cr
				\mbox{analysis and forecast of $f$}\end{array}}
		\\
		\textcolor{blue}{\widehat\x_{k+1}=\DDS(\widehat\x_k)}
		\arrow[r, draw=blue, swap,"\textcolor{blue}{\mbox{exact } {\displaystyle \Koop f=f\circ \DDS}}"]
		&
		\textcolor{blue}{\begin{array}{l}\mbox{analysis and} \cr \mbox{forecast of } f(\widehat\x_k)\end{array}}
		\arrow[u,draw=magenta, leftrightarrow,swap,"\textcolor{magenta}{\mbox{small corrections}}",dashed]
	\end{tikzcd}
	\]
	\caption{\label{FIG:Goal-BSS} Commutative diagram of backward shadowing stability of \edmd{}/\kmd{}: The computed \edmd{}/\kmd{} analysis of an observable $f$ corresponds, up to small correction, to the exact analysis of $f$ along a trajectory with slightly changed initial condition.}
\end{figure}
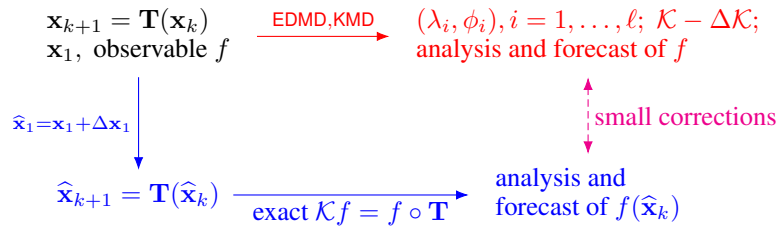
\noindent To that end, in this section we provide necessary auxiliary technical results. \S \ref{SS=ConceptualFramework} presents the key idea how to trade $\Delta\Koop$ for a perturbation of the trajectory, and 
\S \ref{SS=NumerShadowing} revisits the shadowing theory, whose combination with the backward error analysis of the \edmd{} is the key
for the ultimate goal -- the commutative diagram in Figure \ref{FIG:Goal-BSS}.
%

\subsection{Introductory ideas}\label{SS=ConceptualFramework}
Suppose that the observable of interest, $f$, is subject to the \kmd{}, i.e. 
let $f(\x_k)\approx \sum_j \alpha_j \phi_j(\x_1)\lambda_j^{k-1}$. This representation is used for forecasting so that $f(\x_{k+1})$ is predicted as $\sum_j \alpha_j \phi_j(\x_1)\lambda_j^{k}$, which is the result of 
\begin{eqnarray}
	f(\x_{k+1})&=&(\Koop^k f)(\x_1)=(\Koop f)(\x_k) \approx \sum_j \alpha_j \lambda_j^{k-1}(\Koop\phi_j)(\x_1) =  \sum_j \alpha_j \lambda_j^{k-1}(\lambda_j\phi_j(\x_1)+\zbr_j(\x_1)) \nonumber \\ 
	&\approx& \sum_j \alpha_j \phi_j(\x_1)\lambda_j^{k} = ((\Koop-\Delta\Koop)\sum_j \alpha_j \phi_j\lambda_j^{k-1})(\x_1) = ((\Koop-\Delta\Koop)^k\sum_j \alpha_j \phi_j\lambda_j)(\x_1).
	\label{eq:K-DK-powers}
\end{eqnarray}
Essentially, the \kmd{} and its applications (such as forecasting) are in terms of the spectral data of $\Koop-\Delta\Koop$.

Backward stability is a matter of interpretation, and it is most useful if it can be expressed in terms of the original problem, which is evaluation and analysis of observables along trajectories of the dynamical system defined by the mapping $\DDS$. 
Hence, it is important to pursue along the backward error analysis principles and to try to push $\Delta\Koop$ one more step backward -- into the map $\DDS$ of the given dynamical system and even further into the initial condition.

To explore this possibility, let $f:\mathcal{X}\subseteq\R^n\longrightarrow \R$ be an observable of interest, assumed smooth.\footnote{Similar analysis, \emph{mutatis mutandis}, applies if $f$ is defined on a smooth manifold.} Then
$$
((\Koop-\Delta\Koop)f)(\zbx) = (f\circ\DDS)(\zbx) + e(\zbx),\;\;e(\zbx) = - (\Delta\Koop f)(\zbx)= - \sum_{j=1}^\ell\langle f,\widetilde\phi_j\rangle \zbr_j(\zbx),
$$ 
and the goal is to interpret $e(\zbx)$ as a result of perturbing the system mapping $\DDS$.
To that end, let $\bftau_f(\zbx)$ be small and such that 
$$
f(\DDS(\zbx)+\bftau_f(\zbx)) \approx f(\DDS(\zbx)) + \nabla f(\DDS(\zbx))^T\bftau_f(\zbx) = f(\DDS(\zbx)) + e(\zbx) .
$$
Then, assuming $\DDS(\zbx)$ is not critical point of $f$, $\Koop-\Delta\Koop$ acts on $f$ in a composition fashion as 
\begin{equation}\label{eq:tauf(x)}
	((\Koop-\Delta\Koop)f)(\zbx) \approx f(\DDS(\zbx)+\bftau_f(\zbx))\;\mbox{with}\;
	\bftau_f(\zbx) = \frac{e(\zbx)}{\|\nabla f(\DDS(\zbx))\|_2^2}\nabla f(\DDS(\zbx)).
\end{equation}
This can be made more precise by including  the second order term in Taylor expansion, or by invoking the implicit function theorem.  Taking one step further, 
$$
\DDS(\zbx+\delta\zbx) = \DDS(\zbx) + [D{\DDS}(\zbx)]\delta\zbx + O(\|\delta\zbx\|),
$$
where (assuming nonsingular Jacobian  $[D{\DDS}(\zbx)]$) gives $\DDS(\zbx + [D{\DDS}(\zbx)]^{-1}\bftau_f(\zbx))\approx \DDS(\zbx)+\bftau_f(\zbx)$. 
Altogether, 
\begin{eqnarray}
	((\Koop-\Delta\Koop)f)(\zbx) &\approx& f(\DDS(\zbx)+\bftau_f(\zbx))\approx f(\DDS(\zbx + [D{\DDS}(\zbx)]^{-1}\bftau_f(\zbx))) \nonumber \\
	&=& f(\DDS(\zbx + \delta\zbx))=(\Koop f)(\zbx + \delta\zbx),\;\;\delta\zbx =  [D{\DDS}(\zbx)]^{-1}\bftau_f(\zbx) .\label{eq:tauf(x)-2}
\end{eqnarray}
The above relations are derived rather informally, but, more importantly, (\ref{eq:tauf(x)}), (\ref{eq:tauf(x)-2}) indicate that it may be possible to turn the backward error $\Delta\Koop$ into perturbation of the data, while keeping the original composition operator attached to the mapping $\DDS$. Clearly, using only (\ref{eq:tauf(x)}), (\ref{eq:tauf(x)-2}) to track the errors along a long trajectory of (\ref{eq:DDS}) and pushing them backward into the data snapshots is practically not feasible, and more powerful techniques are needed.
In fact, the lines above are at the core of the proof of the shadowing lemma \cite{Anosov1970}, \cite{Bowen1975}, \cite{Palmer-Shadowing-2000}. Here, the setting is more complicated and it requires a composition of the shadowing theory with the evaluation of observables along trajectories expressed by the Koopman operator and its perturbation.

It follows immediately, that the question of backward stability of the \edmd{} belongs to the framework of the shadowing theory in dynamical systems \cite{Palmer-Shadowing-2000}, \cite{book-Pilyugin-shadowing},   \cite{Pilyugin2011}, \cite[Chapter 18]{KatokHasselblatt1995}, \cite{NusseYorke1988}. The idea of shadowing is illustrated in Figure \ref{FIG:Shadow1}.

\begin{figure}[H]
	\begin{center}
		\begin{tikzpicture}[>=Stealth,scale=0.89,transform shape]
			
			\coordinate (x0) at (0.0,0.10);
			\coordinate (x1) at (1.0,1.00);
			\coordinate (x2) at (2.0,0.35);
			\coordinate (x3) at (3.0,1.55);
			\coordinate (x4) at (4.0,0.70);
			\coordinate (x5) at (5.0,1.75);
			\coordinate (x6) at (6.0,1.05);
			
			\draw[blue,densely dashed]
			(x0)--(x1)--(x2)--(x3)--(x4)--(x5)--(x6);
			
			\foreach \p in {x0,x1,x2,x3,x4,x5,x6}
			\fill[blue] (\p) circle (2.2pt);
			
			\node[blue] at (-0.09,-0.09) {$\x_1$} ;
			\coordinate (z0) at (0.12,-0.08);
			\coordinate (z1) at (1.10,0.82);
			\coordinate (z2) at (2.12,0.52);
			\coordinate (z3) at (3.10,1.34);
			\coordinate (z4) at (4.10,0.90);
			\coordinate (z5) at (5.10,1.55);
			\coordinate (z6) at (6.10,1.25);
			
			\draw[green!60!black,densely dashed]
			(z0)--(z1)--(z2)--(z3)--(z4)--(z5)--(z6);
			
			\foreach \p in {z0,z1,z2,z3,z4,z5,z6}
			\filldraw[green!60!black,rotate=45]
			(\p) rectangle +(0.09,0.09);
			
			\node[green!60!black] at (0.49,-0.10) {$\widehat\x_1$} ;
			\coordinate (y0) at (-0.10,0.30);
			\coordinate (y1) at (0.88,1.20);
			\coordinate (y2) at (1.90,0.15);
			\coordinate (y3) at (2.88,1.75);
			\coordinate (y4) at (3.90,0.48);
			\coordinate (y5) at (4.88,1.95);
			\coordinate (y6) at (5.88,0.85);
			
			\draw[red,densely dashed]
			(y0)--(y1)--(y2)--(y3)--(y4)--(y5)--(y6);
			
			\foreach \p in {y0,y1,y2,y3,y4,y5,y6}
			\filldraw[red] (\p) rectangle +(0.09,0.09);
			
			\node[red] at (-0.13,0.59) {$\widetilde\x_1$} ;
			\foreach \a/\b in {
				y0/z0,
				y1/z1,
				y2/z2,
				y3/z3,
				y4/z4,
				y5/z5,
				y6/z6}
			{
				\draw[gray!70,->] (\a)--(\b);
			}
			
			\node[blue] at (8.8,1.05) {exact trajectory, $\x_{k+1}=\DDS(\x_k)$};
			\node[green!60!black] at (9.99,1.6) {shadowing trajectory, $\widehat\x_{k+1}=\DDS(\widehat\x_k)$, $\|\widehat\x_k-\widetilde\x_k\|_{\R^n}\leq\varepsilon$};
			\node[red] at (9.77,0.65) {pseudo--trajectory, $\|\widetilde\x_{k+1}-\DDS(\widetilde\x_k)\|_{\R^n} \leq \delta=\max_k\delta_k$};
			
		\end{tikzpicture}
	\end{center}
	\caption{\label{FIG:Shadow1} $(\widetilde\x_k)_{k=1}^{M+1}$ is a $\delta$--pseudo--trajectory of (\ref{eq:DDS}) along which the observable $f$ evolves as (\ref{eq:K-DKkfx1}). It is $\varepsilon$--shadowed by a true orbit $\widehat\x_{k+1}=\DDS(\widehat\x_k)$ if $\max_k \|\widehat\x_k-\widetilde\x_k\|_{\R^n}\leq\varepsilon$. In the ideal case, the forecast (\ref{eq:K-DK-powers}) based on the inexact eigenpairs of $\Koop$ is exact for the values of $f$ along a nearby 
		trajectory $(\widehat\x_k)_{k=1}^{M+1}$.}
\end{figure}
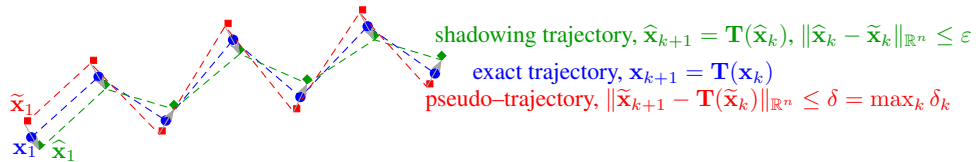
\subsection{Numerical shadowing}\label{SS=NumerShadowing}
The backward error analysis of the \edmd{} from \S \ref{SS=IndivResidAggregBackErr} and the informal discussion in \S \ref{SS=ConceptualFramework} indicate that the residuals in the \edmd{} can be transformed through the backward perturbation $\Delta\Koop$ into the trajectory data, and that the shadowing theory provides rigorous mathematical apparatus. 
\S \ref{SSS=AuxLemma}, \S \ref{SSS=ShadowCOntactBackStable} revisit the shadowing theory \cite[Chapter 11]{Palmer-Shadowing-2000} and provide an extension of the numerical shadowing lemma.
\subsubsection{Auxiliary lemma -- review and an extension}\label{SSS=AuxLemma}
This section provides the key technical result for the numerical shadowing -- Lemma \ref{Lemma:Palmer++}, which is an extension of \cite[Lemma 11.4]{Palmer-Shadowing-2000}. Detailed proof is provided for the sake of clarity, and in particular to show how the assertion can be strengthened using techniques from numerical linear algebra. The following notation will be used.

Let $\zbF:\Omega\subseteq\R^p\longrightarrow \R^q$ be of class $C^2$ on the open set $\Omega$, 
$\zbF(\zbx)=\begin{pmatrix} \zbF_1(\zbx), \ldots, \zbF_q(\zbx)\end{pmatrix}^T$. 
The first and the second derivative at $\zbx\in\Omega$,  
$D\zbF(\zbx): \R^p\longrightarrow\R^q$ and $D^2\zbF(\zbx): \R^p\times \R^p \longrightarrow\R^q$,
are used for local approximation of $\zbF$ as 
\begin{eqnarray}
	&& \zbF(\zbx+\Delta\zbx) = \zbF(\zbx) + D\zbF(\zbx)\Delta\zbx + \frac{1}{2}D^2\zbF(\zbx)(\Delta\zbx,\Delta\zbx) + o(\|\Delta \zbx\|_{\R^p}^2) , \nonumber\\
	&& \| \zbF(\zbx+\Delta\zbx) -  \zbF(\zbx) - D\zbF(\zbx)\Delta\zbx \|_{\R^q} \leq 
	\frac{1}{2}\sup_{\|\zby-\zbx\|_{\R^p}\leq \|\Delta\|_{\R^p}} \|D^2\zbF(\zby)\|\|\Delta\zbx\|_{\R^p}^2 , \label{eq:Taylor-F-2}
\end{eqnarray}
where the bilinear operator $D^2\zbF(\zbx)(\zbu,\zbv)$ acts as 
$$
\R^p\times \R^p \ni (\zbu,\zbv) \mapsto D^2\zbF(\zbx)(\zbu,\zbv) = \begin{pmatrix} \zbu^T D^2\zbF_1(\zbx)\zbv\cr \vdots\cr 
	\zbu^T D^2\zbF_q(\zbx)\zbv\end{pmatrix} \in \R^q ,
$$
and its norm is in the pair of vector norms $\|\cdot\|_{\R^p}$, $\|\cdot\|_{\R^q}$ defined by 
$$
\mabs{D^2\zbF(\zbx)}_{p,q} = \sup_{\zbu, \zbv\neq\0} \frac{\|D^2\zbF(\zbx)(\zbu,\zbv)\|_{\R^q}}{\|\zbu\|_{\R^p}\|\zbv\|_{\R^p}} .
$$
For an operator $\zbA : \R^p\longrightarrow \R^q$, $\|\zbA\|_{p,q}$ denotes the norm compatible with  $\|\cdot\|_{\R^p}$, $\|\cdot\|_{\R^q}$: $\|\zbA \zbx\|_{\R^q}\leq \|\zbA\|_{pq} \|\zbx\|_p$. 

\begin{lemma}\label{Lemma:Palmer++}
	Let $\zbF$ be as above, and let $\zby\in\Omega$ be such that $\|\zbF(\zby)\|_{\R^q}\leq \delta$ with some $\delta > 0$, and let $D\zbF(\zby)$ be of full row rank. Choose any right inverse $\zbG_\zby$	of $D\zbF(\zby)$ ($D\zbF(\zby) \zbG_\zby=\Id_q$),  and set $\varepsilon = 2 \delta \|\zbG_y\|_{q,p}$ and 
	\begin{equation}\label{eq:D2F-My}
		M(\zbF;\zby,\varepsilon) = \sup\{ \mabs{D^2\zbF(\zbx)}_{p,q} : \zbx\in \overline{B(\zby,\varepsilon)} \},\;\;\overline{B(\zby,\varepsilon)}=\{\zbx\in\R^p : \| \zbx - \zby\|_{\R^p} \leq \varepsilon\} .
	\end{equation}
	If $\overline{B(\zby,\varepsilon)} \subset\Omega$, and $M(\zbF;\zby,\varepsilon) \leq 1/(\|\zbG_\zby\|_{q,p}\varepsilon)$, then the equation $\zbF(\zbx)=\0$ has a solution $\zbx\in \overline{B(\zby,\varepsilon)}$. Furthermore,
	\begin{itemize}
		\item The right inverse can be selected so that in the case $p>q$ the solution $\zbx$ satisfies $\zbx_{i_j}=\zby_{i_j}$, $j=1, \ldots, p-q$. 
		\item In particular, if $D\zbF(\zby)$ is a full spark matrix, then for any selection of $p-q$ indices $i_1,\ldots, i_{p-q}$, 
		there is $\zbx\in\zbF^{-1}(\{\0\})\cap \overline{B(\zby,\varepsilon)}$ that coincides  $\zby$ at $i_1,\ldots, i_{p-q}$. 
	\end{itemize}
\end{lemma}
{\sc Proof} Following \cite{Palmer-Shadowing-2000}, the strategy is to rephrase the equation $\zbF(\zbx)=\0$ as the fixed point problem $\zbH(\zbx)=\zbx$, with the auxiliary function $\zbH(\zbx)=\zby - \zbG_\zby (\zbF(\zbx)-D\zbF(\zby)(\zbx-\zby))$. Indeed, $\zbH(\zbx)=\zbx$
	can be expressed as
	\begin{equation}\label{eq:Hx-x}
		\zbx - \zby = -\zbG_y (\zbF(\zbx) - D\zbF(\zby)(\zbx-\zby)),
	\end{equation}
	and applying $D\zbF(\zby)$ from the left yields (because $\zbG_\zby$ is a right inverse) $\zbF(\zbx)=\0$.
	
	The Brouwer fixed point theorem will assure existence of the fixed point of $\zbH$ in $\overline{B(\zby,\varepsilon)}$ if $\zbH ( \overline{B(\zby,\varepsilon)} ) \subseteq \overline{B(\zby,\varepsilon)}$. To check that, take $\zbx\in\overline{B(\zby,\varepsilon)}$ and estimate, using (\ref{eq:Taylor-F-2}), (\ref{eq:D2F-My}), 
	\begin{eqnarray*}
		\| \zbH(\zbx) - \zby\|_{\R^p} &\leq& \|\zbG_y\|_{q,p} \|\zbF(\zbx) - \zbF(\zby) - D\zbF(\zby)(\zbx-\zby) + \zbF(\zby)\|_{\R^q} \\
		&\leq& \|\zbG_y\|_{q,p} ( \frac{1}{2} M(\zbF;\zby,\varepsilon) \|\zbx-\zby\|_{\R^p}^2 + \delta) \\
		&\leq& \|\zbG_y\|_{q,p} M(\zbF;\zby,\varepsilon) \frac{\varepsilon^2}{2} + \|\zbG_y\|_{q,p}\delta \leq \varepsilon,\;\;\mbox{i.e.}\;\;
		\zbH(\zbx)\in \overline{B(\zby,\varepsilon)}.
	\end{eqnarray*}
	Hence, there is $\zbx\in \overline{B(\zby,\varepsilon)}$ such that $\zbH(\zbx)=\zbx$ and $\zbF(\zbx)=\0$. 
	
	The relation $D\zbF(\zby) \zbG_\zby=\Id_q$ implies that $D\zbF(\zby)$ is necessarily of full row rank (its left nullspace must be trivial), and as already noted $\zbG_\zby$ is of full column rank. This means that the assumption on the existence of the right inverse contains tacit assumption that $q\leq p$. If $q=p$, then $\zbG_y=D\zbF(\zby)^{-1}$. But, if $q<p$ then two immediate observations are that $\zbF$ cannot be injective (hence, it allows multiple solutions to $\zbF(\zbx)=\0$) and that there are infinitely many right inverses of $D\zbF(\zby)$. The immediate choice is the Moore--Penrose generalized inverse, $\zbG_y=D\zbF(\zby)^\dagger$, which is of minimal Frobenius norm among all possible choices.
	In this context, particularly interesting are sparse inverses---those with many zero entries--- that can be constructed as follows. Since $D\zbF(\zby)$ is of full row rank, there is permutation matrix $\Pi$ such that the leading $q\times q$ submatrix of $D\zbF(\zby)\Pi$ is nonsingular. Now, the condition $D\zbF(\zby)\zbG_y=\Id_q$ is satisfied with 
	$$
	\zbG_\zby=\Pi \begin{pmatrix} [D\zbF(\zby)\Pi(:,1:q)]^{-1} \cr \0\end{pmatrix}.
	$$
	Applying $\Pi^T$ to the relation (\ref{eq:Hx-x}), where $\zbx$ is the fixed point,  yields
	\begin{equation}\label{eq:Hx-x-2}
		\Pi^T(\zbx - \zby) = - \begin{pmatrix} [D\zbF(\zby)\Pi(:,1:q)]^{-1} \cr \0\end{pmatrix}(\zbF(\zbx) - D\zbF(\zby)(\zbx-\zby)),
	\end{equation}
	i.e. $\zbx$ and $\zby$ coincide at those indices recorded in $\Pi(:,q+1:p)$. The permutation $\Pi$ is determined using the \qr{} factorization with column pivoting,
	$$
	D\zbF(\zby) \Pi = \zbQ \begin{pmatrix} \zbR_{[1]} & \zbR_{[2]}\end{pmatrix},\;\;\zbR_{[1]}\in\R^{q\times q},\;\zbR_{[2]}\in\R^{q\times (p-q)},\;
	\zbQ^T\zbQ = \zbQ\zbQ^T=\Id_q .
	$$
	The pivoting permutation $\Pi$ is determined dynamically with the task to bring the most important linearly independent columns of $D\zbF(\zby)$ to the leading $q$ positions in
	$D\zbF(\zby) \Pi$, optimally with the inverse of small norm. In practice, the most frequently used is the  Businger--Golub pivoting \cite{bus-gol-65}, which guarantees that $\zbR_{[1]}$ has a strong form of diagonal dominance: $|(\zbR_{[1]})_{ii}|\geq \|\zbR_{[1]}(i:j,j)\|_2$, $j=i,\ldots, q$. Theoretically the best---but NP hard---is the volume maximizing pivoting \cite{Knuth-SemiOptBases-1985}, \cite{Goreinov19971}, \cite{mgu-eis-96}, \cite{Massei-maxVol-2021} that ensures maximal matrix volume of $D\zbF(\zby)\Pi(:,1:q)$ among all $q\times q$ submatrices of $D\zbF(\zby)$. It should be noted that in this construction, 
	each $\Pi$ yields an inverse $\zbG_\zby$ with its norm $\|\zbG_\zby\|$ that is a parameter in the construction of $\zbx$.
	Hence, the matching indices (where $\Pi^T(\zbx-\zby)=\0$) are any $q$ column indices $i_1,\ldots, i_{p-q}$ such that deleting the corresponding columns from $D\zbF(\zby)$ leaves a
	$q\times q$ nonsingular matrix. 
\hfill$\boxtimes$

\begin{remark}
	{\em
		In the space of $q\times p$ matrices with $p>q$ the set of full spark matrices (all $q\times q$ submatrices nonsingular) is open and dense in $\zbR^{q\times p}$ (it is open in the Zariski topology). The classical example of full spark matrix is a $q\times p$ Vandermonde matrix with $p$ distinct nodes.
		With $\zby$ fixed, the full spark property of $D\zbF(\zby)$ is generic in the compact--open topology of $C^1(\Omega,\R^q)$.
	}
\end{remark}

\subsubsection{Shadowing with contacts}\label{SSS=ShadowCOntactBackStable}
We now prove an extension of the numerical shadowing lemma \cite[Theorem 11.3]{Palmer-Shadowing-2000} that in addition to the existence of a shadowing true trajectory to a pseudo--trajectory asserts that certain number of contacts of the two trajectories is possible.

The proof relies on Lemma \ref{Lemma:Palmer++}, and to facilitate its direct application, trajectories will be reshaped into column vectors using the $\vecop(\cdot)$ operator, e.g.
a trajectory $\y=\begin{pmatrix} \zby_1 & \ldots & \zby_{M+1}\end{pmatrix}$ will be identified with the vector $\vecop(\y)$ in $\R^p$, $p={(M+1)n}$, $\vecop(\y)=\begin{pmatrix} \zby_1^T & \ldots & \zby_{M+1}^T\end{pmatrix}^T$.  This isomorphism is used to define the norm $\|\vecop(\y)\|_{\R^p}=\max_{k=1:M+1}\|\zby_k\|$, where $\|\cdot\|$ is the norm in $\R^n$ in which $\y$ is $\delta$--pseudo--trajectory. In the same way, trajectory with $M$ elements is a vector in $\R^q$ with $q=M n$; the norm $\|\cdot\|_{\R^q}$ is defined analogously. These transformations of trajectories into column vectors are just for convenience and for a direct application of Lemma \ref{Lemma:Palmer++} in the proof of a new shadowing result. Without loss of rigor, the notation can be simplified by identifying $\vecop(\y)\equiv\y$. The components of block--partitioned $\y$ in this notation are $\zby_k=\y((k-1)n+1:kn)$, $k=1,\ldots, M+1$.

A trajectory $\x$ of (\ref{eq:DDS}) can be conveniently described as the zero of the function $\zbF: \R^p\longrightarrow \R^q$ defined by
$$
\zbF(\x)\equiv \zbF(\left(\begin{smallmatrix} \zbx_1 \cr \vdots\cr \zbx_k \cr \vdots\cr \zbx_{M+1}\end{smallmatrix}\right)) = \left( \begin{smallmatrix} \zbx_2-\DDS(\zbx_1) \cr \vdots \cr \zbx_{k+1}-\DDS(\zbx_k) \cr \vdots \cr \zbx_{M+1}-\DDS(\x_M)\end{smallmatrix}\right),
$$
and a $\delta$--pseudo--trajectory $\y$ is characterized simply by $\|\zbF(\y)\|_{\R^q}\leq \delta$.
\begin{theorem}\label{TM:palmer++}
	Let $\DDS : \R^n\longrightarrow\R^n$ be of class $C^2$ and let $\y=(\zby_1,\ldots,\zby_{M+1})\in\R^{n\times (M+1)}$ be its $\delta$--pseudo--trajectory, $\max_k \|\zby_{k+1}-\DDS(\zby_k)\|\leq \delta$. Choose a right inverse $\zbG_\y$ of $D\zbF(\y)$ and set $\varepsilon=2\|\zbG_\y\|_{q,p}\delta$, and 
	\begin{equation}\label{eq:TM:My}
		M(\DDS;\y,\varepsilon) = \sup \{ \mabs{D^2\DDS(\zbx)}_{p,q} : \zbx\in \bigcup_{k=1}^{M+1} \overline{B(\zby_k,\varepsilon)}\}.
	\end{equation}
	Assume that $M(\DDS;\y,\varepsilon)\|\zbG_\y\|_{q,p} \varepsilon \leq 1$.
	Then:
	\begin{enumerate} 
		\item The pseudo--trajectory $\y$ is $\varepsilon$--shadowed by a true trajectory $\x=(\zbx_1,\ldots,\zbx_{M+1})$ of $\DDS$.
		\item There is a trajectory $\x$ that $\varepsilon$--shadows $\y$  and, moreover, $\x$ and $\y$ coincide at at  least $n$ positions. In particular, 
		\begin{itemize}
			\item There is always a true trajectory $\x$ such that $\zbx_1=\zby_1$. 
			\item If all Jacobians $D\DDS(\zby_1), \ldots, D\DDS(\zby_M)$ are nonsingular, then there is a true trajectory $\x$ such that $\zbx_{M+1}=\zby_{M+1}$.
		\end{itemize}
	\end{enumerate}
\end{theorem}
{\sc Proof}
	The Jacobian and the Hessian of $\zbF$ are identified from
	$$
	\zbF(\x+\delta\x) \! =\! \begin{pmatrix} \zbx_2+\delta\zbx_2-\DDS(\zbx_1) - D\DDS(\zbx_1)\delta\zbx_1 - D^2\DDS(\zbx_1)(\delta\zbx_1,\delta\zbx_1) + o(\|\delta\zbx_1\|^2)\cr \vdots \cr \zbx_{k+1}+\delta\zbx_{k+1}-\DDS(\zbx_k)-D\DDS(\zbx_k)\delta\zbx_k - D^2\DDS(\zbx_k)(\delta\zbx_k,\delta\zbx_k) + o(\|\delta\zbx_k\|^2)\cr \vdots \cr \zbx_{M+1}+\delta\zbx_{M+1}-\DDS(\x_M) - D\DDS(\x_M)\delta\x_M - D^2\DDS(\zbx_M)(\delta\zbx_M,\delta\zbx_M) + o(\|\delta\zbx_M\|^2)\end{pmatrix}
	$$
	as 
	\begin{eqnarray}
		D\zbF(\x) &=& \begin{pmatrix} 
			-D\DDS(\zbx_1) & \Id_n & & & \cr 
			& -D\DDS(\zbx_2) & \ddots &  & & \cr 
			& & \ddots & \Id_n & \cr 
			& & & -D\DDS(\zbx_M) & \Id_n
		\end{pmatrix}, \label{eq:DF-D2F}\\
		D^2\zbF(\x)(\delta\x,\delta\x)&=&\begin{pmatrix} D^2\DDS(\zbx_1)(\delta\zbx_1,\delta\zbx_1) \cr 
			D^2\DDS(\zbx_2)(\delta\zbx_2,\delta\zbx_2) \cr
			\vdots \cr
			D^2\DDS(\zbx_M)(\delta\zbx_M,\delta\zbx_M)\end{pmatrix}. \label{eq:DF-D2F-b}
	\end{eqnarray}
	By the assumption, $\|\zbF(\y)\|_q \leq \delta$. Clearly, $\x\in\R^p$ is a trajectory of $\DDS$ if and only if $\zbF(\x)=\0$.
	The assertion of the theorem is that
	$\zbF$ has zero in an $\varepsilon$ ball centered at $\y$, so the proof reduces to an application of Lemma \ref{Lemma:Palmer++}. 
	
	To that end, let $\zbG_{\y}$ be a right inverse of $D\zbF(\y)$ and $\varepsilon= 2 \|\zbG_{\y}\|_{q,p} \delta$. Lemma \ref{Lemma:Palmer++} guarantees existence of $\x\in \overline{B(\y,\varepsilon)}$ such that $\zbF(\x)=\0$ if $M(\zbF;\y,\varepsilon)\|\zbG_{\y}\|_{q,p}\varepsilon\leq 1$, and it remains to see how 
	$M(\DDS;\y,\varepsilon)$ given in (\ref{eq:TM:My}) relates to $M(\zbF;\y,\varepsilon)$. First, note that the definition of the norm implies that $\x\in \overline{B(\y,\varepsilon)}$ is equivalent to $\zbx_k\in \overline{B(\zby_k,\varepsilon)}$ for all $k$. 
	The formulas (\ref{eq:DF-D2F},\ref{eq:DF-D2F-b}) and the definition of the norm imply
	$$
	\mabs{D^2\zbF(\x)} \leq \max_k \mabs{D^2\DDS(\zbx_k)} = 
	\mabs{D^2\DDS(\zbx_{k_*})} \leq \sup\{ \mabs{D^2\DDS(\zbx)} : \zbx\in \overline{B(\zby_{k_*},\varepsilon)} \},
	$$
	and thus $M(\zbF;\y,\varepsilon)\leq M(\DDS;\y,\varepsilon)$. Lemma \ref{Lemma:Palmer++} applies and $\y$ is $\varepsilon$--shadowed by $\x$.
	
	To prove the second part, first note that $D\zbF(\y)$ given in (\ref{eq:DF-D2F}) is $q\times p\equiv Mn\times (M+1)n$ of full row rank $Mn$. Further, independent of $\y$, the submatrix $D\zbF(\y)(1:q,n+1:p)$ is $q\times q$ nonsingular:
	\begin{eqnarray*}
		D\zbF(\y)(1:q,n+1:p) &=& \left(\begin{smallmatrix} 
			\Id_n & & & & \cr
			-D\DDS(\zby_2) & \Id_n &  & & \cr
			& -D\DDS(\zby_3) & \Id_n & \cr
			&  & \ddots & \ddots & \cr
			&  & & -D\DDS(\zby_M) & \Id_n
		\end{smallmatrix}\right), \\ 
		(D\zbF(\y)(1:q,n+1:p)^{-1})_{[ij]} &=&\left\{ \begin{array}{lc}
			\Id_n, & i=j \cr
			\0, & i < j	\cr
			D\DDS(\zby_i)D\DDS(\zby_{i-1})\cdots D\DDS(\zby_{j+1}),  & i > j
		\end{array}\right. .
	\end{eqnarray*}
	(Here $(\cdot)_{[ij]}$ denotes the $(i,j)$th element in the block partition.) For example, in the $4\times 4$ block partition,
	$$
	D\zbF(\y)(1:q,n+1:p)^{-1}=
	\left( \begin{smallmatrix} 
		\Id_n & \0 & \0 & \0 \cr
		D\DDS(\zby_2) & \Id_n & \0 & \0 \cr
		D\DDS(\zby_3)D\DDS(\zby_2) & 	D\DDS(\zby_3) & \Id_n & \0\cr
		D\DDS(\zby_4)D\DDS(\zby_3)D\DDS(\zby_2) & D\DDS(\zby_4)D\DDS(\zby_3) & D\DDS(\zby_4) & \Id_n
	\end{smallmatrix}\right).
	$$
	Hence, by the second part of Lemma \ref{Lemma:Palmer++}, it is possible to have $\zbF(\x)=\0$ (a true trajectory) with $\zbx_1=\zby_1$. (Recall that by the definition of  $\vecop(\cdot)$, $\x(1:n)=\zbx_1$, $\y(1:n)=\zby_1$.) 
	
	Now, consider removing the last $n$ columns of $D\zbF(\y)$:
	\begin{eqnarray*}
		&& D\zbF(\y)(1:q,p-n)=\left(\begin{smallmatrix}
			-D\DDS(\zby_1) & \Id_n & & & \cr
			& -D\DDS(\zby_2) & \Id_n & & \cr
			&  & \ddots & \ddots &  \cr
			&  &  & -D\DDS(\zby_{M-1}) & \Id_n \cr
			&  &  &  & -D\DDS(\zby_M)
		\end{smallmatrix}\right),\\
		&&\det(D\zbF(\y)(1:q,p-n)) = (-1)^{nM}\prod_{k=1}^M \det(D\DDS(\zby_k)) .
	\end{eqnarray*} 
	If this matrix is singular, then its inverse is 
	$$
	(D\zbF(\y)(1:q,p-n)^{-1})_{[ij]} = 
	\left\{ \begin{array}{lc}
		-(D\DDS(\zby_i))^{-1}(D\DDS(\zby_{i+1}))^{-1}\cdots (D\DDS(\zby_{j}))^{-1},  & i \leq j \cr
		\0, & i > j	\cr
	\end{array}\right. .
	$$
	In the case $M=3$ this reads
	$$
	D\zbF(\y)(1\! :\! q,p-n)^{-1} \!=\! \left(\!\begin{smallmatrix}
		-(D\DDS(\zby_1))^{-1} & -(D\DDS(\zby_1))^{-1}(D\DDS(\zby_2))^{-1} & 
		-	(D\DDS(\zby_1))^{-1}(D\DDS(\zby_2))^{-1}(D\DDS(\zby_3))^{-1} \cr
		\0 & -(D\DDS(\zby_2))^{-1} & -(D\DDS(\zby_2))^{-1}(D\DDS(\zby_3))^{-1} \cr 
		\0 & \0 & -(D\DDS(\zby_3))^{-1}
	\end{smallmatrix}\!\right)\! .
	$$
	It should be noted that selecting the intersection points adds a constraint that directly influences the norm of the right inverse $\zbG_\zby$ and thus the shadowing parameter $\varepsilon$.
\hfill$\boxtimes$

\section{Backward shadowing stability of \kmd{}/\edmd{}}\label{S=MixedBackShadowEDMD}
This section formalizes the discussion from \S \ref{SS=ConceptualFramework} in terms of the shadowing theory and proposes a numerical analysis framework for assessing accuracy of \kmd{}/\edmd{}. The goal is to turn the residuals $\zbr_i=\Koop\phi_i-\lambda_i\phi_i$ of the computed eigenpairs $(\lambda_i,\phi_i)$, $i=1,\ldots, \ell$, into a perturbation in terms of the original system (\ref{eq:DDS}). The first step is replacement of $\Koop$ with $\Koop-\Delta\Koop$ as in Theorem \ref{TM:DeltaKoop}. Next, the ideas from \S \ref{SS=ConceptualFramework} are formalized in terms of trajectories of (\ref{eq:DDS}).

Consider evaluation of the Krylov sequence $f, (\Koop-\Delta\Koop)f, (\Koop-\Delta\Koop)^2f, \ldots$ 
where the observable of interest $f$ is assumed backward stable in the following sense: \\ 

\begin{mdframed}
	\textbf{Backward stability assumption for $f$}: For a state $\zbx$ and small error $e(\zbx)$ in evaluating $f(\zbx)$, there exist small backward error $b(\zbx)$ such that $f(\zbx)+e(\zbx)=f(\zbx+b(\zbx))$. (For a construction of $b(\zbx)$ see (\ref{eq:tauf(x)}).)
\end{mdframed}
\ \\
Now, compute (for some $\widetilde\x_1$, including the case $\widetilde\x_1=\x_1$)
$$
((\Koop-\Delta\Koop)f)(\widetilde\x_1) = f(\DDS(\widetilde\x_1)) - (\Delta\Koop f)(\widetilde\x_1) = f(\DDS(\widetilde\x_1)) + e(\widetilde\x_1) = f(\DDS(\widetilde\x_1) + b(\widetilde\x_1)) ,
$$
and define $\widetilde\x_2=\DDS(\widetilde\x_1) + b(\widetilde\x_1)$, so that $f(\widetilde\x_2)=((\Koop-\Delta\Koop)f)(\widetilde\x_1)$. In the next step, 
$$
((\Koop-\Delta\Koop)^2f)(\widetilde\x_1)=((\Koop-\Delta\Koop)f)(\widetilde\x_2) = f(\DDS(\widetilde\x_2)) + e(\widetilde\x_2) = f(\DDS(\widetilde\x_2) + b(\widetilde\x_2)) = f(\widetilde\x_3),
$$
where $\widetilde\x_3=\DDS(\widetilde\x_2) + b(\widetilde\x_2)$.
Hence, once the computed eigenpairs are accepted and used in the \kmd{},
$\Koop$ is replaced with $\Koop-\Delta\Koop$ and instead of the
trajectory $\x_1, \x_2=\DDS(\x_1), \ldots, \x_{k+1}=\DDS(\x_k), \ldots$ 
the observable $f$ is evaluated along the pseudo--trajectory 
\begin{equation}\label{eq:pseudo-trajectory}
	\widetilde\x_1=\x_1, \;\;\widetilde\x_2 = \DDS(\widetilde\x_1) + b(\widetilde\x_1), \;\; \widetilde\x_3=\DDS(\widetilde\x_2)+b(\widetilde\x_2), \;\;\ldots,\;\; \widetilde\x_{k+1}=\DDS(\widetilde\x_k)+b(\widetilde\x_k), \ldots 
\end{equation}
where $b(\widetilde\x_k)$ depends on the backward stability of $f$ at $\DDS(\widetilde\x_k)$ and on the forward perturbation
\begin{equation}\label{eq:e(xk)}
	e(\widetilde\x_k) = - (\Delta\Koop f)(\widetilde\x_k) 
	= -\sum_{j=1}^\ell \langle f, \widetilde\phi_j\rangle \zbr_j(\widetilde\x_k) .
\end{equation}
Of course, of interest is convergent \edmd{}, so that with $N\rightarrow\infty$, $M\rightarrow\infty$, selection of $\ell$ eigenpairs with small residuals is feasible. \\

\begin{mdframed}
	\textbf{Numerical convergence assumption for \edmd{}}: Assume that $M$ and $N$ are large enough, and that $\Delta\Koop$ is built using $\ell$ eigenpairs with small residuals so that all $|e(\widetilde\x_k)|\leq O(\vertiii{\Delta\Koop})$, and that, based on the assumption on backward stability of $f$ and (\ref{eq:e(xk)}) (See Proposition \ref{PROP:DKAdjoint}.), all  $\|b(\widetilde\x_k)\|_{\R^n}\leq O(|e(\widetilde\x_k)|)$.
\end{mdframed}
\ \\
Note that this corresponds to evaluating the Krylov sequence 
\begin{eqnarray}
	&& f(\widetilde\x_2)=((\Koop-\Delta\Koop)f)(\widetilde\x_1),\;\; f(\widetilde\x_3)=((\Koop-\Delta\Koop)f)(\widetilde\x_2)=((\Koop-\Delta\Koop)^2f)(\widetilde\x_1),\;\; \ldots, \nonumber\\
	&& \ldots, \; f(\widetilde\x_{k+1})=((\Koop-\Delta\Koop)^k f)(\widetilde\x_1), \;\ldots \label{eq:K-DKkfx1}
\end{eqnarray}
The question is whether the iterations (\ref{eq:K-DKkfx1}) and the approximate spectral decomposition (\ref{eq:K-DK-powers}) can be interpreted as nearly analyzing/forecasting the observable $f$ along a true trajectory of $\DDS$ that is close to the pseudo--trajectory  
$(\widetilde\x_k)_{k=1}^{M+1}$ (\ref{eq:pseudo-trajectory}), where (in some suitable norm $\|\cdot\|_{\R^n}$ in $\R^n$)
\begin{equation}\label{eq:pseudo-traj-tildexk}
	\|\widetilde\x_{k+1}-\DDS(\widetilde\x_k)\|_{\R^n} \leq \delta_k,\;\;\delta_k=\|b(\widetilde\x_k)\|_{\R^n}.
\end{equation}
Theorem \ref{TM:palmer++} asserts that (\ref{eq:pseudo-traj-tildexk}) is $\varepsilon$--shadowed by a true trajectory $\widehat\x_{k+1}=\DDS(\widehat\x_k)$, $\|\widehat\x_k-\widetilde\x_k\|_{\R^n}\leq\varepsilon$ for all $k$; see Figure \ref{FIG:Shadow1}. Note that
$$
f(\widetilde\x_k) = f(\widehat\x_k) + \nabla f(\widehat\x_k)^T(\widetilde\x_k-\widehat\x_k) + o(\|\widehat\x_k-\widetilde\x_k\|_{\R^n}) = f(\widehat\x_k) + \Delta f(\widehat\x_k).
$$
Hence, $\varepsilon$ small perturbations of the values in the sequence (\ref{eq:K-DKkfx1}) will make it an exact evaluation of the observable $f$ along a nearby (shadowing) true trajectory $(\widehat\x_k)_{k=1}^{M+1}$ of the system. We call this property (\emph{mixed}) \emph{backward shadowing stability}, with the corresponding  commutative diagram shown in Figure \ref{FIG:diagramBackShadow}.
\vspace{-4mm}
\begin{figure}[H]
	\[
	\tikzcdset{
		arrows={line width=1pt, -{Stealth[length=2mm,width=1.2mm]}}
	}
	\begin{tikzcd}[column sep=huge,row sep=large]
		{\color{red}(\widetilde\x_k)_{k=1}^{M+1}}
		\arrow[r,"{(\Koop-\Delta\Koop)f}"]
		\arrow[d,"\varepsilon\mbox{{\small--shadow}}"']
		&
		f(\widetilde\x_1), \ldots, f(\widetilde\x_{M+1})
		\\
		{\color{green!60!black}(\widehat\x_k)_{k=1}^{M+1}}
		\arrow[r,"{\Koop f}"]
		&
		f(\widehat\x_1), \ldots, f(\widehat\x_{M+1})
		\arrow[u,shift left = 11, dashed,bend left=25, "{+\Delta f(\widehat\x_1), \ldots, }"{pos=0.5, right}]
		\arrow[u, shift right=8,dashed,bend left=15, "{+\Delta f(\widehat\x_{M+1}) }"{pos=0.5, right}]
		&
	\end{tikzcd}
	\]
	\vspace{-2mm}
	\caption{\label{FIG:diagramBackShadow} \emph{Commutative diagram for (mixed) backward shadowing stability} of \edmd{}: If approximate eigenpairs of $\Koop$ are used in the analysis and forecasting of the dynamics, the computation corresponds, up to small forward error, to a pseudo--trajectory that is shadowed by a true trajectory of the original system.}
\end{figure}
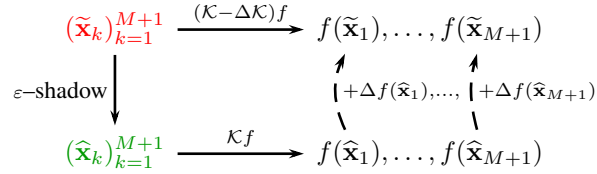
\begin{remark}
	{\em
		The adjective \emph{mixed} is used in numerical analysis of algorithms to emphasize that the computed output after small perturbation corresponds to exact computation with backward changed input. Many algorithms in numerical linear algebra that are called backward stable actually have this slightly weaker form of mixed stability. For example, the \qr{} factorization can compute only numerically orthogonal factor, so the computed factorization is not the \qr{} factorization. However, small correction of the computed orthogonal factor can make it exactly orthogonal and with this correction we obtain the exact \qr{} factorization of backward perturbed input matrix.
	}
\end{remark}

\noindent In this way, the errors due to using the computed (inexact) spectral data of $\Koop$ are expressed in more intrinsic way, inherent to  the original system (\ref{eq:DDS}). In particular, if the system is measure preserving, then it can be claimed that the computed numerical solution corresponds to the same measure preserving system, but along a nearby trajectory. In a data driven setting, where the input data are contaminated by noise, this proposed backward shadowing is a natural framework for numerical analysis because it pertains directly to perturbation of the data defining the problem. If the backward shadowing is tight and comparable to the uncertainty in the input, then the computed output is as good as one can hope for.   

In an ideal case, the situation can be illustrated as shown in Figure \ref{FIG:Shadow1}.
By the second part of Theorem \ref{TM:palmer++}, it is possible to select $\widehat\x_1=\widetilde\x_1$; since $\widetilde\x_1=\x_1$, the nearby trajectory can be the original one. However, as already noted in the proof of the theorem, this constraint may impact $\varepsilon$, and that another trajectory matches the pseudo--trajectory better. The accuracy of the actual output depends on the forward stability of the mapping $\DDS(\cdot)$. See Figure \ref{FIG:Shadow2}.

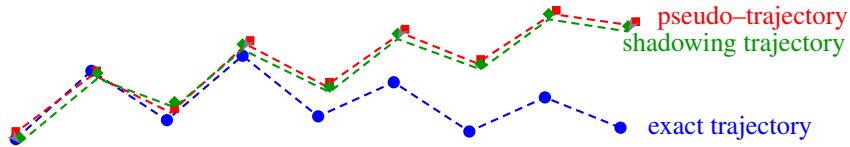
\begin{figure}[H]
	\begin{center}
		\begin{tikzpicture}[>=Stealth]
			
			\coordinate (x0) at (0.0,0.20);
			\coordinate (x1) at (1.0,1.10);
			\coordinate (x2) at (2.0,0.45);
			\coordinate (x3) at (3.0,1.30);
			\coordinate (x4) at (4.0,0.50);
			\coordinate (x5) at (5.0,0.95);
			\coordinate (x6) at (6.0,0.30);
			\coordinate (x7) at (7.0,0.75);
			\coordinate (x8) at (8.0,0.35);
			
			\draw[blue,densely dashed,thick]
			(x0)--(x1)--(x2)--(x3)--(x4)--(x5)--(x6)--(x7)--(x8);
			
			\foreach \p in {x0,x1,x2,x3,x4,x5,x6,x7,x8}
			\fill[blue] (\p) circle (2.3pt);

			\coordinate (y0) at (-0.05,0.25);
			\coordinate (y1) at (1.02,1.05);
			\coordinate (y2) at (2.05,0.55);
			\coordinate (y3) at (3.05,1.45);
			\coordinate (y4) at (4.10,0.90);
			\coordinate (y5) at (5.10,1.60);
			\coordinate (y6) at (6.10,1.20);
			\coordinate (y7) at (7.10,1.85);
			\coordinate (y8) at (8.15,1.70);
			
			\draw[red,densely dashed,thick]
			(y0)--(y1)--(y2)--(y3)--(y4)--(y5)--(y6)--(y7)--(y8);
			
			\foreach \p in {y0,y1,y2,y3,y4,y5,y6,y7,y8}
			\filldraw[red] (\p) rectangle +(0.10,0.10);

			\coordinate (z0) at (0.05,0.15);
			\coordinate (z1) at (1.08,1.00);
			\coordinate (z2) at (2.10,0.62);
			\coordinate (z3) at (3.00,1.38);
			\coordinate (z4) at (4.15,0.82);
			\coordinate (z5) at (5.05,1.52);
			\coordinate (z6) at (6.15,1.12);
			\coordinate (z7) at (7.05,1.78);
			\coordinate (z8) at (8.10,1.62);
			
			\draw[green!60!black,densely dashed,thick]
			(z0)--(z1)--(z2)--(z3)--(z4)--(z5)--(z6)--(z7)--(z8);
			
			\foreach \p in {z0,z1,z2,z3,z4,z5,z6,z7,z8}
			\filldraw[green!60!black,rotate=45]
			(\p) rectangle +(0.10,0.10);

			\foreach \a/\b in {
				y0/z0,
				y3/z3,
				y5/z5,
				y8/z8}
			{
				\draw[gray,->] (\a)--(\b);
			}
			
			\node[blue] at (9.45,0.35) {exact trajectory};
			\node[red] at (9.75,1.80) {pseudo--trajectory};
			\node[green!60!black] at (9.50,1.45) {shadowing trajectory};
			
			
		\end{tikzpicture}	
	\end{center}
	\caption{\label{FIG:Shadow2} It is possible that $\DDS$ in (\ref{eq:DDS}) is forward unstable and  the pseudo--trajectory drifts away from the original trajectory, but remains close to another true trajectory with slightly changed initial condition.}	
\end{figure}
\noindent 
This discussion aims to establish, with minimum of technical details, a new framework for the numerical analysis of the computational \emph{koopmanism}, in particular of the \edmd{}/\kmd{}. 
More detailed technical analysis will have to take into account the following issues from numerical linear algebra, operator theory and dynamical systems theory:

\noindent \emph{(i)} In (\ref{eq:pseudo-traj-tildexk}), the norm $\|\cdot\|_{\R^n}$ is not specified. Although all norms on $\R^n$ are equivalent, the choice of suitable norm might be advantageous from the practical point of view.\\  
\emph{(ii)} For more detailed and concrete analysis and more in depth exploration of the shadowing theory, convergence and backward error analysis of the \edmd{}, and spectral perturbation theory of closed linear operators in the computational \emph{koopmanism}, it is necessary to specify the domain of $\Koop$ and the underlying Hilbert space structure. \\
\emph{(iii)} Here the mapping $\DDS$ is assumed as in Theorem \ref{TM:palmer++}, and for more general statements the analysis needs to be done in the framework of the dynamical systems theory. 

The last part of this work, \S \ref{S=DataDrivenResids}, addresses the practical question of computing
the norm of $\Delta\Koop$ in a purely data driven scenario.

\section{Data driven residuals}\label{S=DataDrivenResids}
In a purely data driven scenario, the residuals can only be approximated, based on the supplied data, and this section shows how to compute the residual norms
$\|\zbr_i\|$  and $\vertiii{\Delta\Koop}$. 
Assuming the $L^2$ structure\footnote{See \S \ref{S=Preliminaries}.} in $\FF$, the norm of $\zbr_i$ is 
\begin{equation}\label{eq:Kphi-residual}
	\frac{\|\zbr_i\|}{\|\phi_i\|} = \frac{\|\Koop\phi_i - \lambda_i\phi_i\|}{\|\phi_i\|} = \frac{(\int_{\mathcal{X}} |(\Koop\phi_i)(\zbx) - \lambda_i\phi_i(\zbx)|^2 d\mu(\zbx))^{1/2}}{(\int_{\mathcal{X}} |\phi_i(\zbx)|^2d\mu(\zbx))^{1/2}},
\end{equation}
and $\vertiii{\Delta\Koop}$ is approximated using Proposition \ref{PROP:DeltaKnorm}. The individual residuals can be used to discard the spurious eigenpairs and the aggregated residuals of the selected ones, whose norm can be estimated using (\ref{eq:normDK-estimate}), allows for a backward stability analysis in the sense of the discussion in \S \ref{SEC:ResidsBackwardpert} and \S \ref{S=MixedBackShadowEDMD}.

\begin{proposition}\label{PROP-residual-new}
	For an approximate eigenpair $(\lambda_i,\phi_i)$, the residual (\ref{eq:Kphi-residual}) can be estimated as 	
	\begin{equation}\label{eq-eigfun-3}
		\|\Koop \phi_i - \lambda_i\phi_i\| \approx \| \zbW^{1/2} (\OY\bfss_i - \lambda_i\OX\bfss_i)\|_2,
	\end{equation}
	where 
	$$
	\|\phi_i\|=\|\GramM^{1/2} [\phi_i]_{\mathcal B}\|_2\approx \| \zbW^{1/2}\OX \bfss_i\|_2.
	$$
	Let $\bfR$ be the matrix obtained by stacking the residuals column--wise,
	$
	\bfR(:,i) = \lambda_i \OX\bfss_i - \OY\bfss_i,\; i=1,\ldots, \ell,
	$	
	and let $\widetilde\Phi_x$ the matrix of tabulated values of $\widetilde\phi_1, \ldots, \widetilde\phi_\ell$,  as explained in \S \ref{SSS=tildephis}. Then
	\begin{equation}\label{eq:normDK-estimate}
		\vertiii{\Delta\Koop} \approx \sqrt{\lambda_{\max}((\widetilde\Phi_x^*\zbW \widetilde\Phi_x) (\bfR^*\zbW \bfR))}=\sigma_{\max}(\sqrt{\bfR^*\zbW \bfR}\widetilde\zbL),
	\end{equation}
	where $\widetilde\zbL$ is the lower triangular Cholesky factor of $\widetilde\Phi_x^*\zbW \widetilde\Phi_x$.
\end{proposition}
{\sc Proof}
	Since $\phi_i$ is given by the dictionary as $\phi_i(\zbx)=\Basef(\zbx)[\phi_i]_{\mathcal B}$, $[\phi_i]_{\mathcal B}=\bfss_i$, the formula for $\|\phi_i\|$ follows.
	The residual $\rho_i$ evaluated at the data reads\footnote{See \S \ref{SS=AprroxEigenpairs}.}
	$\begin{pmatrix} \rho_i(\x_1), \ldots, \rho_i(\x_M)\end{pmatrix}^T=(\OX\KoopM_N-\OY)(:,i)$, and (\ref{eq-eigfun-2}) implies
	\begin{equation}\label{eq:ri(x1-xM)}
		\begin{pmatrix}\bfr_i(\x_1) & \ldots & \bfr_i(\x_M)
		\end{pmatrix}^T = (\OX\KoopM_N - \OY)\bfss_i = \lambda_i \OX\bfss_i - \OY\bfss_i.
	\end{equation}
	Hence
	\begin{eqnarray*}
		\|\Koop \phi_i - \lambda_i\phi_i\|^2 &\approx& \sum_{k=1}^M w_k |(\Koop\phi_i)(\x_k)-\lambda_i\phi_i(\x_k)|^2 = \sum_{k=1}^M w_k |\bfr_i(\x_k)|^2 \\ 
		&=& \| \zbW^{1/2} (\OY\bfss_i - \lambda_i\OX\bfss_i)\|_2^2.
	\end{eqnarray*}
	Then, similar reasoning as above yields 
	\begin{equation}\label{eq:R*WR}
		(\bfR^*\zbW \bfR)_{ij} = \sum_{k=1}^M w_k \overline{\bfr_i(\x_k)}\bfr_j(\x_k) \approx \langle \zbr_j,\zbr_i\rangle = (\zbM_{\zbR})_{ij} .
	\end{equation}
	Let $\widetilde\phi_i(x)=\Basef(x) \widetilde\bfss_i$ (or $\widetilde\phi_i(x)=\Basef(x) \widehat\bfss_i$). Then
	\begin{equation}\label{eq:Phi*WPhi}
		(\widetilde\Phi_x^*\zbW \widetilde\Phi_x)_{ij} = \sum_{k=1}^M w_k\overline{\widetilde\phi_i(\x_k)}\widetilde\phi_j(\x_k) \approx \langle \widetilde\phi_j,\widetilde\phi_i\rangle = \widetilde\GramM_{ij}.
	\end{equation}
	The proof completes by using (\ref{eq:R*WR}), (\ref{eq:Phi*WPhi})
	in Proposition \ref{PROP:DeltaKnorm}. If $\widetilde\Phi_x^*\zbW \widetilde\Phi_x=\widetilde\zbL \widetilde\zbL^*$ is the Cholesky factorization, then $\widetilde\zbL$ can be computed as the adjoint of the upper triangular factor in the \qr{} factorization of $\zbW^{1/2}\widetilde\Phi_x$. Similarly, the semidefinite square root
	$\sqrt{\bfR^*\zbW \bfR}$ can be replaced with the Cholesky factor of 
	$\bfR^*\zbW \bfR$ computed using the \qr{} factorization of $\zbW^{1/2}\bfR$.
\hfill$\boxtimes$

\begin{remark}
	{\em
		The formula (\ref{eq-eigfun-3}) was derived in \cite{Colbrook-Townsend-Rigor-2024} by first writing 
		\begin{eqnarray*}
			\|\Koop\phi_i-\lambda_i\phi_i\|^2 &=& \langle \Koop\phi_i-\lambda_i\phi_i,\Koop\phi_i-\lambda_i\phi_i\rangle 
			\\
			&=&\langle \Koop\phi_i,\Koop\phi_i\rangle - \lambda_i\langle \phi_i,\Koop\phi_i\rangle - \overline{\lambda}_i\langle \Koop\phi_,\phi_i\rangle + |\lambda_i|^2 \langle\phi_i,\phi_i\rangle \\
			&=& \bfss_i^*\left[ (\langle \Koop\basef_j,\Koop\basef_k\rangle)_{j,k=1}^n - 
			\lambda_i (\langle \basef_j,\Koop\basef_k\rangle)_{j,k=1}^n \right. 
			\\  
			&& \left. - \overline{\lambda}_i (\langle \Koop\basef_j,\basef_k\rangle)_{j,k=1}^n 
			+ |\lambda_i|^2 \GramM \right]\bfss_i ,
		\end{eqnarray*}
		and then noticing that for large $M$ and convergent cubature formula $(\OY^*\zbW \OY)_{ij} \approx \langle \Koop\basef_j,\Koop\basef_i\rangle$.
		This yields
		\begin{eqnarray*}
			\|\Koop\phi_i-\lambda_i\phi_i\|^2 &\approx& \bfss_i^* [\OY^*\zbW \OY -\lambda_i (\OX^*\zbW\OY)^* - \overline{\lambda}_i \OX^*\zbW\OY + |\lambda_i|^2 \OX^*\zbW\OX]\bfss_i \\
			&=& \| \zbW^{1/2} (\OY\bfss_i - \lambda_i\OX\bfss_i)\|_2^2 .
		\end{eqnarray*}
		Although equivalent to this, the proof of (\ref{eq-eigfun-3}) in Proposition \ref{PROP-residual-new} more intuitively connects with the construction 	of the eigenpairs explained in \S \ref{SS=Compresionmatrix} and \S \ref{SS=AprroxEigenpairs}.
	}
\end{remark}

\section{Concluding remarks and future research}\label{S=ConcludingRem}
This work proposes a conceptual framework of backward shadowing stability analysis of the \edmd{}/\kmd{} in the context of data driven Koopman operator based analysis of nonlinear dynamics. We believe that this perspective will open a number of promising avenues for future research in development and numerical analysis of methods for computational data driven analysis in the Koopman operator framework. An immediate research direction from the dynamical systems theory point of view  is e.g. to specialize the proofs and technical details to dynamical systems and data/initial conditions that allow shadowing (e.g. Anosov systems on compact manifolds, utilizing the exponential dichotomies etc.).  Further, it would be beneficial to formulate the concept of backward shadowing stability for kernel \edmd{} and the RKHS spaces, and for computation of the Liouville operator \cite{RosenfeldKamalapurkarGrussJohnson2022}. How the gained insights improve our understanding of the \edmd{} (including its variations such as the measure preserving \edmd{} \cite{Colbrook-mpEDMD-2023} and the Koopman--Schur decomposition \cite{Drmac-Mezic-Koopman-Schur-2026}) and its applicability in difficult scenarios? 
How the problem of vanishing residuals \cite{Colbrook-AnotherLook-2024} affects the backward shadowing analysis? 
Some of these themes are already addressed in our ongoing research \cite{Drmac-Djimoti-BackShadow-2026}.

\newcommand{\noop}[1]{}

\end{document}